\documentclass{tran-l}
\usepackage{amsmath, amssymb}
\usepackage{mathrsfs}
\usepackage{tikz-cd}
\usepackage{lineno}
\usepackage{hyperref}
\usepackage{xcolor}

\newtheorem{theorem}{Theorem}[section]
\newtheorem{lemma}[theorem]{Lemma}

\newtheorem{corollary}[theorem]{Corollary}
\newcommand{\powerset}{\raisebox{.15\baselineskip}{\Large\ensuremath{\wp}}}
\theoremstyle{definition}
\newtheorem{definition}[theorem]{Definition}
\newtheorem{examples}[theorem]{Examples}
\newtheorem{example}[theorem]{Example}

\theoremstyle{remark}
\newtheorem{remark}[theorem]{Remark}

\numberwithin{equation}{section}

\begin{document}

\title{ON SEMI-TRANSITIONAL AND TRANSITIONAL RINGS}

\author{Sourav Koner}
\address{Department of Mathematics, The University of Burdwan, Burdwan Rajbati, West Bengal 713104}
\email{harakrishnaranusourav@gmail.com}
 
\author{Titas Saha}
\address{Department of Mathematics, The University of Burdwan, Burdwan Rajbati, West Bengal 713104}
\email{titassaha26@gmail.com}

\author{Biswajit Mitra}
\address{Department of Mathematics, The University of Burdwan, Burdwan Rajbati, West Bengal 713104}
\email{bmitra@math.buruniv.ac.in}


\begin{abstract}
In this paper, we introduce and study two new classes of commutative rings, namely semi-transitional rings and transitional rings, which extend several classical ideas arising from rings of continuous functions and their variants. A general framework for these rings is developed through the notion of semi-transition and transition maps, leading to a systematic exploration of their algebraic and topological properties. Structural results concerning product rings, localizations, and pm-rings are established, showing that these new classes naturally generalize familiar examples such as $k[x]$, $C^*(X)$, and the ring of admissible ideal convergent real sequences. Ideals and filters induced by semi-transition maps are analyzed to characterize prime and maximal ideals, revealing a duality between algebraic and set-theoretic constructions. Furthermore, conditions under which semi-transitional rings become semiprimitive are determined, and a Stone–Čech-like compactification is constructed for transitional rings, giving rise to a new perspective on unique extension properties in topological algebra.
\end{abstract}

\maketitle


\section{Introduction}

The study of rings of continuous functions and their structural generalizations has long served as a bridge between algebra and topology. Classical objects such as $C(X)$ and $C^{*}(X)$ have motivated the exploration of numerous ring-theoretic analogues through which algebraic operations capture topological properties of underlying spaces. In this work, we develop a broader theory by introducing two new classes of rings---\emph{semi-transitional} and \emph{transitional rings}---that subsume and generalize several known constructions while preserving key interactions between algebraic ideals and topological zero-sets.

The framework begins with the definition of a \emph{semi-transition map} on a ring, encoding how zero-sets behave under addition and multiplication. This notion captures familiar examples such as polynomial rings over a field and rings of bounded continuous functions. We then demonstrate that many structural features, including localization and product formation, are preserved under semi-transition maps. Further, by imposing additional separation conditions on the underlying structure, we define \emph{transitional rings}, which exhibit stronger correspondence with pm-rings and admit a rich topology on their maximal spectra.

The paper proceeds as follows. Section~\ref{sec:semi} introduces semi-transitional rings and illustrates fundamental examples and constructions. Section~\ref{sec:ideals} investigates ideals and the associated filter structures, establishing a bijective correspondence between maximal ideals and ultrafilters. Section~\ref{sec:structure} provides structural characterizations of semiprimitive associated semi-transitional rings. Section~\ref{sec:transitional} extends the framework to transitional rings and offers examples such as the ring of $I$-convergent sequences of reals which is discussed in section \ref{sec: ideal conv}. Section~\ref{sec:omega} develops an $\Omega$-compactification analogous to the Stone--\v{C}ech compactification, highlighting the unique extension property within this new context. Finally, Section~\ref{sec:application} presents an application of the $\Omega$-compactification, demonstrating how the Stone--\v{C}ech and Banaschewski compactifications can be derived from its unique extension property. 

\section{Semi-transitional rings}\label{sec:semi}

\begin{definition}\label{SITARAM}
A ring $A$ is said to be a semi-transitional ring if there exists a nonempty set $\Lambda$ and a map $\psi$ that associates for each $a \in A$ a subset $\psi(a)$ of $\Lambda$ such that\\
$(i)$ $\psi(ab) = \psi(a) \cup \psi(b)$, for all $a, b \in A$;\\
$(ii)$ $\psi(1) = \emptyset$; $\psi(a) = \Lambda$ if and only if $a = 0$;\\
$(iii)$ $\psi(a) \cap \psi(b) = \psi (c) \subseteq \psi(a + b)$, for all $a, b \in A$ and for some $c \in (a, b)$; moreover, if $\psi(a) \cap \psi(b) = \emptyset$, then $(a^{m}, b^{n})$ is principal, for some $m, n \in \mathbb{N}$.\\
Any map $\psi$ with these three properties is called a semi-transition map on $A$.
\end{definition}

\begin{examples}\label{SITARAM1}
The most important examples of semi-transitional rings, from our standpoint, are the polynomial ring $k[x]$ over a field $k$ and the ring $C^{\ast}(X)$ of real-valued bounded continuous functions on a topological space $X$; the semi-transition maps $\psi_{1}: k[x] \rightarrow \powerset(k)$ and $\psi_{2}: C^{\ast}(X) \rightarrow \powerset(X)$ are defined by 
\begin{align}
\psi_{1}(f) = \{\alpha \in k \mid f(\alpha) = 0\}, \psi_{2}(g) = \{x \in X \mid g(x) = 0\}.
\end{align}
It is not difficult to see that the maps $\psi_{1}$ and $\psi_{2}$ satisfy the conditions $(i)$ and $(ii)$ of Definition \eqref{SITARAM}. To see that they also satisfy the condition $(iii)$ of Definition \eqref{SITARAM}, observe that since $k[x]$ is a B\'ezout domain, therefore for any $f_1$, $f_2 \in K[x]$, $$(f_{1}, f_{2}) = (f_{3}).$$ 

So, we have the following two relations among $f_1, f_2$ and $f_3$. 

\begin{equation}\label{bh}
f_3 = g f_1 + h f_2
\end{equation}

\text{and}

\begin{equation}\label{mol}
\begin{cases}
f_1 = f_3 u \\
f_2 = f_3 v
\end{cases}
\end{equation}

for $g,h,u,v \in k[x]$. From \eqref{bh}, we get
$$\psi_1 (f_1) \cap \psi_1 (f_2) \subseteq  \psi_1 (f_3)$$
and 
from \eqref{mol}, we get
$$\psi_1 (f_3) \subseteq \psi_1 (f_1) \cap \psi_1 (f_2).$$ Patching both of them, we evidently have

\begin{align}
\psi_{1}(f_{1}) \cap \psi_{1}(f_{2}) = \psi_{1}(f_{3}) \subseteq \psi_{1}(f_{1} + f_{2}), \text{where } (f_{1}, f_{2}) = (f_{3}).
\end{align}

Remaining part of $(iii)$ Definition \ref{SITARAM}, follows trivially again due to B\'ezout nature of $k[x]$ with $n = m = 1$.

In case of $\psi_2$, the next expression \ref{bm} trivially follows from $(1.10)$ in \cite{gj60}

\begin{align}\label{bm}
\psi_{2}(g_{1}) \cap \psi_{2}(g_{2}) = \psi_{2}(g_{1}^{2} + g_{2}^{2}) \subseteq \psi_{2}(g_{1} + g_{2}).
\end{align}

Now, if $\psi_{2}(g_{1}) \cap \psi_{2}(g_{2}) = \emptyset$, then $g_{1}^{2} + g_{2}^{2} \neq 0$, for all $x \in X$. Since both the functions $\frac{g_{1}^{2}}{g_{1}^{2} + g_{2}^{2}}$ and $\frac{g_{2}^{2}}{g_{1}^{2} + g_{2}^{2}}$ are bounded and continuous, so $g_{1}^{2} = h_{1}(g_{1}^{2} + g_{2}^{2})$ and $g_{2}^{2} = h_{2}(g_{1}^{2} + g_{2}^{2})$, for some $h_{1}, h_{2} \in C^{\ast}(X)$. This gives that $(g_{1}^{2}, g_{2}^{2}) = (g_{1}^{2} + g_{2}^{2})$, here $n=m=2$.
\end{examples}

\begin{examples}\label{SITARAM2}
Another important examples of semi-transitional rings are the ring $C_{c}^{\ast}(X)$ of real-valued bounded continuous functions on a topological space $X$ with countable range and the ring $C= \{f\in C^*(\mathbb{N}) : \lim_{n\to\infty} f(n) \text{ exists}\}$ of real-valued convergent sequences; the semi-transition maps $\psi_{1}: C_{c}^{\ast}(X) \rightarrow \powerset(X)$ and $\psi_{2}: C \rightarrow \powerset(\mathbb{N}\cup \{p\})$, where $p \notin \mathbb{N}$, are defined by

$$\psi_{1}(f) = \{x \in X \mid f(x) = 0\}$$ 

and

$$
\Psi_2(f) = 
\begin{cases} 
Z(f) & \text{if } \lim_{n \to \infty} f(n) \ne 0 \\
Z(f) \cup \{p\} & \text{if } \lim_{n \to \infty} f(n) = 0
\end{cases}
$$

It can be verified easily that the maps $\psi_{1}$ and $\psi_{2}$ satisfy the conditions $(i)$ and $(ii)$ of Definition \eqref{SITARAM}. To prove that they also satisfy the condition $(iii)$ of Definition \eqref{SITARAM}, notice that
\begin{align}
\psi_{1}(f_{1}) \cap \psi_{1}(f_{2}) = \psi_{1}(f_{1}^{2} + f_{2}^{2}) \subseteq \psi_{1}(f_{1} + f_{2}).
\end{align}
Moreover, if $\psi_{1}(f_{1}) \cap \psi_{1}(f_{2}) = \emptyset$, then a similar argument as in Example \eqref{SITARAM1} gives that $(f_{1}^{2}, f_{2}^{2}) = (f_{1}^{2} + f_{2}^{2})$.

If $f, g \in C$ and $h = f^2 + g^2$ then, it is not difficult to prove case-wise that,
\begin{align}
\psi_{2}(f) \cap \psi_{2}(g) = \psi_{2}(h) \subseteq \psi_{2}(f + g).
\end{align}
Further, if $\psi_{2}(f) \cap \psi_{2}(g) = \emptyset$, then we may consider the following two cases: 

\textbf{Case 1:} $lim_{n\to \infty} f(n) \neq 0 $ and $lim_{n\to \infty} g(n) \neq 0$. 

\textbf{Case 2:} Either of $lim_{n\to \infty} f(n)$ or $;lim_{n\to \infty} g(n)$ equals to $0$.

As $\psi_{2}(f) \cap \psi_{2}(g) = \emptyset$, by definition, Case 1 does not hold.

For Case 2, $\psi_{2}(f) \cap \psi_{2}(g) = Z(f^2 + g^2) = \emptyset$. This implies that $\frac{f^2}{f^2+g^2}$ and $\frac{g^2}{f^2+g^ 2}$ are in $C$. The rest follows along the same line of argument as in Example \eqref{SITARAM1} to conclude that $(f^{2}, g^{2}) = (f^2 + g^2)$. Here also $n = m = 2$.
\end{examples}

The natural inquiry regarding the existence of infinitely many transitional rings is affirmatively addressed in the following theorem. 

\begin{theorem}\label{SITARAM3}
Let $F$ be a nonempty index set and let $\psi_{\alpha}: A_{\alpha} \rightarrow \powerset(\Lambda_{\alpha})$ be a semi-transition map for each $\alpha \in F$. Suppose there exists a fixed positive integer $m$ such that for every $\alpha \in F$, $\psi_{\alpha}(a_{\alpha}) \cap \psi_{\alpha}(b_{\alpha}) = \emptyset$ implies $(a_{\alpha}^{m}, b_{\alpha}^{m})$ is principal. Then the product ring $\prod_{\alpha \in F} A_{\alpha}$ is a semi-transitional ring.
\end{theorem}
\begin{proof}
Let  $\Lambda = \bigsqcup_{\alpha \in F} \Lambda_{\alpha}$, that is, $\Lambda$ is the disjoint union of the sets $\Lambda_{\alpha}$'s. Let $A = \prod_{\alpha \in F} A_{\alpha}$. Define a map $\psi: A \rightarrow \powerset(\Lambda)$ by $\psi((a_{\alpha})_{\alpha \in F}) = \bigsqcup_{\alpha \in F} \psi_{\alpha}(a_{\alpha})$. We now show that $\psi$ is a semi-transition map. It is not hard to see that $\psi$ satisfy the conditions $(i)$ and $(ii)$ of Definition \eqref{SITARAM}. To prove that $\psi$ also satisfies the condition $(iii)$ of Definition \eqref{SITARAM}, consider any two elements $a = (a_{\alpha})_{\alpha \in F}$ and $b = (b_{\alpha})_{\alpha \in F}$ of $A$. Since $\psi(a) \cap \psi(b) = \bigsqcup_{\alpha \in F}(\psi_{\alpha}(a_{\alpha}) \cap \psi_{\alpha}(b_{\alpha})) = \bigsqcup_{\alpha \in F}\psi_{\alpha}(c_{\alpha})$, where $c_{\alpha} = u_{\alpha}a_{\alpha} + v_{\alpha}b_{\alpha}$, for all $\alpha \in F$ and for some $u_{\alpha}, v_{\alpha} \in A_{\alpha}$, therefore if we set $c = (c_{\alpha})_{\alpha \in F}$, $u = (u_{\alpha})_{\alpha \in F}$, and $v = (v_{\alpha})_{\alpha \in F}$, then we see that $c = au + bv$ and $\psi(a) \cap \psi(b) = \psi(c)$. Moreover, as $\psi_{\alpha}(a_{\alpha}) \cap \psi_{\alpha}(b_{\alpha}) \subseteq \psi_{\alpha}(a_{\alpha} + b_{\alpha})$, so we have that $\psi(c) \subseteq \psi(a + b)$. Now, suppose that $\psi(a) \cap \psi(b) = \emptyset$. This implies that $\psi_{\alpha}(a_{\alpha}) \cap \psi_{\alpha}(b_{\alpha}) = \emptyset$, for all $\alpha \in F$. Hence, $(a_{\alpha}^{m}, b_{\alpha}^{m}) = (z_{\alpha})$, for some $z_{\alpha} \in A_{\alpha}$ and for all $\alpha \in F$. Now, if we set $z = (z_{\alpha})_{\alpha \in F}$, then we obtain that $(a^{m}, b^{m}) = (z)$.
\end{proof}

Recall that a ring $A$ is said to be a pm-ring if each prime ideal is contained in a unique maximal ideal. It is well-known  that for a ring $A$ with $\mathrm{J}(A) = 0$, $\mathrm{Max}(A)$ is Hausdorff if and only if it is pm; such rings are called soft rings \cite{c82}. The next result shows that every semi-transitional ring can be regarded as a subring of some bigger semi-transitional ring. 

\begin{theorem}\label{SITARAM4}
Every semi-transitional ring $A$ can be embedded into its ring of fractions $S^{-1}A := A_{S}$, for some multiplicatively closed subset $S$ of $A$. Moreover, $A_{S}$ is also a semi-transitional ring.
\end{theorem}
\begin{proof}
Let $\psi: A \rightarrow \powerset(\Lambda)$ be the semi-transition map. Let us now consider the set $S = \{a \in A \mid \psi(a) = \emptyset\}$. It is not difficult to show that $S$ is a multiplicatively closed subset of $A$. Let $A_{S}$ be the localization with respect to $S$. Now, consider the canonical homomorphism $i: A \rightarrow A_{S}$, defined by $i(a) = \frac{a}{1}$. If $\frac{a}{1} = 0$, then $at = 0$, for some $t \in S$, and this implies that $\psi(a) = \Lambda$, that is, $a = 0$. This shows that $i$ is an embedding. \\
$(a)$ Define $\varphi: A_{S} \rightarrow \powerset(\Lambda)$ by $\varphi(\frac{a}{t}) = \psi(a)$, for all $\frac{a}{t} \in A_{S}$. If $\frac{a}{t} = \frac{b}{s}$, for some $a, b \in A$ and for some $s, t \in S$, then $asu = btu$, for some $u \in S$. This implies that $\psi(a) = \psi(b)$. This shows that $\varphi$ is a well-defined map. We now show that $\varphi$ is a semi-transition map on $A_{S}$. Let $b_{1} = \frac{a}{t}$ and $b_{2} = \frac{b}{s}$. \\
$(i)$ Since $\psi(ab) = \psi(a) \cup \psi(b)$, for all $a, b \in A$, so $\varphi(b_{1}b_{2}) = \varphi(b_{1}) \cup \varphi(b_{2})$, for all $b_{1}, b_{2} \in A_{S}$. \\
$(ii)$ Since $\psi(1) = \emptyset$, so $\varphi(\frac{1}{1}) = \emptyset$; now, $\varphi(b_{1}) = \Lambda$, if and only if, $\psi(a) = \Lambda$, if and only if, $a = 0$, if and only if, $b_{1} = 0$. \\
$(iii)$ Since $\psi(a) \cap \psi(b) = \psi(ax + by)$, for some $x, y \in A$, therefore if we set $b_{3} = (\frac{a}{t})\frac{x}{s} + (\frac{b}{s})\frac{y}{t}$, then $\varphi(b_{1}) \cap \varphi(b_{2}) = \varphi(b_{3})$, where $b_{3} \in (b_{1}, b_{2})$. Further, since $\psi(sa + tb) \supseteq \psi(a) \cap \psi(b)$, so $\varphi(b_{1} + b_{2}) \supseteq \varphi(b_{1}) \cap \phi(b_{2})$. 

Assume now that $\varphi(b_{1}) \cap \varphi(b_{2}) = \emptyset$. Then, $\psi(a) \cap \psi(b) = \emptyset$. This implies that $(a^{m}, b^{n}) = (z)$, for some $z \in A$. As $A \cong i(A)$, we get that $(\frac{a^{m}}{1}, \frac{b^{n}}{1}) = (\frac{z}{1})$. Finally, since $(\frac{a^{m}}{t^{m}}, \frac{b^{n}}{s^{n}}) = (\frac{a^{m}}{1}, \frac{b^{n}}{1})$, the result follows. 
\end{proof}

\begin{theorem}\label{SITARAM4.5}
If a semi-transitional ring $A$ is a pm-ring, then\\
$(a)$ $A_{S}$ is also a pm-ring.\\
$(b)$ $\mathrm{Max}(A)$ is homeomorphic to $\mathrm{Max}(A_{S})$.
\end{theorem}
\begin{proof}
$(a)$. Suppose that $\frac{a}{s} + \frac{b}{t} = 1$, for some $\frac{a}{s}, \frac{b}{t} \in A_{S}$. This gives that $(at)u + (bs)u = stu$, for some $u \in S$. Now, applying the conditions $(i)$ and $(iii)$ of Definition \eqref{SITARAM}, we see that $\psi(a) \cap \psi(b) \subseteq \psi(atu + bsu) = \emptyset$. This implies that $(a^m, b^{n}) = (a^mx + b^{n}y)$, for some $x, y \in A$. Thus, there exists $\alpha, \beta \in A$ such that $a^m = \alpha(a^mx + b^{n}y)$ and $b^{n} = \beta(a^{m}x + b^{n}y)$. Further, notice that $\psi(a^{m}x + b^{n}y) \subseteq \psi(a) \cap \psi(b)$. Hence, we obtain that $(a^{m}x + b^{n}y) \in S$. As $\frac{\alpha x}{1} = \frac{a^{m}x}{a^{m}x + b^{n}y}$ and $\frac{\beta y}{1} = \frac{b^{n}y}{a^{m}x + b^{n}y}$, therefore we get that $\alpha x + \beta y = 1$. Since $A$ is a pm-ring, so applying theorem $(4.1)$ in \cite{c82}, we argue that there exists $c, d \in A$ such that $cd = 0$ and $(\alpha x, c) = (\beta y, d) = A$. Since $(\frac{\alpha x}{1}, \frac{c}{1}) \subseteq (\frac{a}{s}, \frac{c}{1})$ and $(\frac{\beta y}{1}, \frac{d}{1}) \subseteq (\frac{b}{t}, \frac{d}{1})$, we obtain that $(\frac{a}{s}, \frac{c}{1}) = (\frac{b}{t}, \frac{d}{1}) = A_{S}$ and $\frac{cd}{1} = 0$, and so applying theorem $(4.1)$ in \cite{c82} again, we conclude that $A_{S}$ is a pm-ring. \\
$(b)$. By theorem $(1.6)$ in \cite{mo71}, it is enough to show that $(M_{1} \cap A) + (M_{2} \cap A) = A$, for distinct maximal ideals $M_{1}$, $M_{2}$ in $A_{S}$. So, suppose that $\frac{a}{s} \in M_{1}$ and $\frac{b}{t} \in M_{2}$ be such that $\frac{a}{s} + \frac{b}{t} = 1$. Then from the first part of this theorem, we have that $\frac{\alpha x}{1} = \frac{a^{m}x}{a^{m}x + b^{n}y}$ and $\frac{\beta y}{1} = \frac{b^{n}y}{a^{m}x + b^{n}y}$, for some $\alpha, \beta, x, y \in A$. Now, since $\alpha x \in M_{1} \cap A$, $\beta y \in M_{2} \cap A$, and $\alpha x + \beta y = 1$, the result follows.
\end{proof}

As an application of Theorem \eqref{SITARAM4.5}, we provide the following examples.

\begin{examples}\label{SITARAM5}
Consider the ring $C^{\ast}(X)$ introduced in Examples \eqref{SITARAM1}. Let $C(X)$ be the ring of real-valued continuous functions on $X$. Though it is well-known that $C(X)$ is a pm-ring and that $\mathrm{Max}(C^{\ast}(X))$ is homeomorphic to $\mathrm{Max}(C(X))$ (see \cite{gj60}), we may conclude this result alternatively in the following way: since $A = C^{\ast}(X)$ is a pm-ring and that $A_{S} = C(X)$, where $S = \{g \in A \mid \psi_{2}(g) = \emptyset\}$, therefore we may apply Theorem \eqref{SITARAM4} to conclude that $A_{S}$ is also a pm-ring and that $\mathrm{Max}(A)$ is homeomorphic to $\mathrm{Max}(A_{S})$.

Now, consider the ring $C$ introduced in Example \eqref{SITARAM2}. Let us consider the subspace $K = \{\frac{1}{m} \mid m \in \mathbb{N}\} \cup \{0\}$ of $\mathbb{R}$ equipped with the usual topology. For each $a = \{x_{n}\}_{n = 1}^{\infty}$ in $C$, we define a function $r_{a}: K \rightarrow \mathbb{R}$ by $r_{a}(\frac{1}{m}) = x_{m}$ and $r_{a}(0) = \lim_{n \to \infty} x_{n}$. It is evident that $r_{a} \in C(K)$, the ring of real-valued continuous functions on $K$. Define $\Psi: C \rightarrow C(K)$ by $\Psi(a) = r_{a}$, for all $a \in C$. It is evident that $\Psi$ is an injective ring homomorphism. Further, for $g \in C(K)$, if we define $x_{m} = g(\frac{1}{m})$, for all $m \in \mathbb{N}$, and $x_{0} = g(0)$, then we have that $\{x_{n}\}_{n = 1}^{\infty} \in C$. This shows that $\Psi$ is surjective. Hence, $C \cong C(K)$. As $K$ is compact, every maximal ideal of $C(K)$ is fixed, that is, $\mathrm{Max}(C) = \{M_{n} \mid n \in \mathbb{N}\} \cup \{M_{0}\}$. 

We now show that $C$ is a pm-ring. On the contrary, assume that $C$ is not a pm-ring. Let $P \in \mathrm{Spec}(C)$. There are two possibilities: for some $l, k \in \mathbb{N}$, either $P \subseteq M_{l} \cap M_{k}$ or $P \subseteq M_{l} \cap M_{0}$. Consider the former case and take any sequence $\{x_{n}\}_{n = 1}^{\infty}$ in $P$. Let us consider the sequences $\{y_{n}\}_{n = 1}^{\infty}$ and $\{z_{n}\}_{n = 1}^{\infty}$ in $C$ such that: $y_{m} = x_{m}$, for all $m \in \mathbb{N} - \{l\}$ and $y_{l} = 1$; $z_{m} = x_{m}$, for all $m \in \mathbb{N} - \{k\}$ and $z_{k} = 1$. Then $\{x_{n}^{2}\}_{n = 1}^{\infty} = \{y_{n}z_{n}\}_{n = 1}^{\infty} \in P$, whereas none of $\{y_{n}\}_{n = 1}^{\infty}$ and $\{z_{n}\}_{n = 1}^{\infty}$ are in $P$, a contradiction, since $P$ is prime. Now, consider the latter case. Take any sequence $\{x_{n}\}_{n = 1}^{\infty}$ in $P$ and define sequences $\{y_{n}\}_{n = 1}^{\infty}$ and $\{z_{n}\}_{n = 1}^{\infty}$ in $C$ as follows: $y_{m} = x_{m}$, for all $m \in \mathbb{N} - \{l\}$ and $y_{l} = 1$;  $z_{m} = 1$, for all $m \in \mathbb{N} - \{l\}$ and $z_{l} = 0$. Observe that $\{x_{n}\}_{n = 1}^{\infty} = \{y_{n}z_{n}\}_{n = 1}^{\infty} \in P$, whereas none of $\{y_{n}\}_{n = 1}^{\infty}$ and $\{z_{n}\}_{n = 1}^{\infty}$ are in $P$, a contradiction, since $P$ is prime. This shows that $C$ is a pm-ring. Let $S =\{\{x_{n}\}_{n = 1}^{\infty} \in C \mid \psi_{2}(\{x_{n}\}_{n = 1}^{\infty}) = \emptyset\}$. Now, if we consider the semi-transitional ring $C_{S}$, then we may apply Theorem \eqref{SITARAM4} to conclude that $C_{S}$ is a pm-ring and that $\mathrm{Max}(C)$ is homeomorphic to $\mathrm{Max}(C_{S})$, in fact they are homeomorphic to the one-point compactification $\mathbb{N}^{\ast}$ of $\mathbb{N}$ (Example \eqref{nkb89}).
\end{examples}

The following examples show that the condition ``$A$ is a pm-ring" in Theorem \eqref{SITARAM4.5} is sufficient but not necessary. 

\begin{example}\label{SITARAMi}
Let $A$ be a B\'ezout domain which is not a local ring. For instance, we may consider $A = \mathbb{Z}$. Clearly, $A$ is not a pm-ring. Consider any nonempty set $\Lambda$. Let us now define a map $\psi: A \rightarrow \powerset(\Lambda)$ by $\psi(0) = \Lambda$ and $\psi(a) = \emptyset$, for all nonzero $a \in A$. It is now routine to check that $\psi$ is a semi-transitional map. Since $S = A - \{0\}$, we see that $A_{S}$ is a field, and so $A_{S}$ is a pm-ring.
\end{example}

\begin{example}\label{SITARAMii}
Consider the semi-transitional ring $A = k[x]$ introduced in Example \eqref{SITARAM1}. If we consider the set $S = \{f \in A \mid \psi_{1}(f) = \emptyset\}$, then Theorem \eqref{SITARAM4} implies that $i: A \rightarrow A_{S}$, defined by $i(a) = \frac{a}{1}$, is an embedding. Now, observe that $i(x)$ is not invertible in $A_{S}$, therefore $A_{S}$ is not a field. This implies that if $M \in \mathrm{Max}(A_{S})$, then $M \cap A$ is a nonzero prime ideal in $A$, and hence $M \cap A$ is a maximal ideal in $A$. This shows that the map $i$ induces the map $i^{\ast}: \mathrm{Max}(A_{S}) \rightarrow \mathrm{Max}(A)$ given by $i^{\ast}(M) = M \cap A$. Now, observe that the following statements are equivalent.

$(a)$ $k$ is algebraically closed.

$(b)$ $i^{\ast}$ is a homeomorphism.\\
To show that $(a) \Rightarrow (b)$, observe that $S = k^{\times}$, since $k$ is algebraically closed. This implies that $A = A_{S}$, so the result follows. Finally, to show $(b) \Rightarrow (a)$, assume on the contrary that $k$ is not algebraically closed. Then there exists an irreducible polynomial $h \in A$ such that $h \in S$. Since $i^{\ast}(M)$ does not intersect $S$, for all $M \in \mathrm{Max}(A_{S})$, so if we consider the maximal ideal $J = (h)$ in $A$, then there does not exist any $M \in \mathrm{Max}(A_{S})$ such that $i^{\ast}(M) = J$, a contradiction. Therefore it must be that $k$ is algebraically closed.  
\end{example}

In Example \eqref{SITARAMii}, if $k$ is not algebraically closed, then though it is clear that the map $i^{\ast}$ is not a homeomorphism, it does not guarantee that there does not exist any homeomorphism between $\mathrm{Max}(A)$ and $\mathrm{Max}(A_{S})$. For example, take $k$ as the field of reals $\mathbb{R}$ in Example \eqref{SITARAMii}. Let $\mathcal{H}$ denote the upper-half plane. We now show that $\mathrm{Max}(A)$ and $\mathrm{Max}(A_{S})$ are homeomorphic. So, at first, consider any $M \in \mathrm{Max}(A)$. Then $M$ is generated by some monic irreducible polynomial in $A$ of degree at most $2$. If $M = (x - \alpha)$, for some $\alpha \in \mathbb{R}$, then $M$ corresponds to the point $\alpha$ of $\mathbb{R} \cup \mathcal{H}$; if $M = ((x - z)(x - \overline{z}))$, for some $z \in \mathcal{H}$, then $M$ corresponds to the point $z$ of $\mathbb{R} \cup \mathcal{H}$. Now, consider any $M \in \mathrm{Max}(A_{S})$. Since $i^{\ast}(M)$ does not intersect $S$, so we get that $i^{\ast}(M) = (x - \beta)$, for some $\beta \in \mathbb{R}$. We now correspond the maximal ideal $M$ in $A_{S}$ to the point $\beta$ of $\mathbb{R}$. Clearly, the correspondence is injective. For surjectivity, let $\beta \in \mathbb{R}$, then the extension of the maximal ideal generated by $(x - \beta)$ in $A$ to $A_S$ is $\{\frac{(x-\beta)f(x)}{g(x)} \mid f(x) \in A \text{ and } g(x) \in S\}$ which is a maximal ideal in $A_{S}$, since it is the kernel of the homomorphism $\frac{l(x)}{m(x)} \mapsto \frac{l(\beta)}{m(\beta)}$ from $A_S$ onto $\mathbb{R}$. Now, since $A$ is a unique factorization domain, every basic closed set in $\mathrm{Max}(A)$ correspond to a finite subset of $\mathbb{R} \cup \mathcal{H}$ and vice versa. So, we conclude that $\mathrm{Max}(A)$ is homeomorphic to $\mathbb{R} \cup \mathcal{H}$ which is endowed with the co-finite topology. A similar argument as above shows that $\mathrm{Max}(A_{S})$ is homeomorphic to $\mathbb{R}$ which is endowed with the co-finite topology. Thus, any bijection from $\mathbb{R}$ to $\mathbb{R} \cup \mathcal{H}$ defines a homeomorphism from $\mathrm{Max}(A_{S})$ to $\mathrm{Max}(A)$.

\section{Ideals and induced filters associated with semi-transition maps}\label{sec:ideals}

For any semi-transition map $\psi: A \rightarrow \powerset(\Lambda)$, we have already seen that the map $\varphi: A_{S} \rightarrow \powerset(\Lambda)$, introduced in Theorem \eqref{SITARAM4}, is again a semi-transition map. Therefore, the ring $A_{S}$ and the map $\varphi$ are called the associated semi-transitional ring attached to $A$ and the associated semi-transition map respectively. Throughout this section, we work with the associated semi-transition map $\varphi$. We shall soon see that there is an interconnection between the ideals in the associated semi-transitional ring $A_{S}$ attached to $A$ and certain filters on $\Lambda$. In order to proceed, we first need the following.

\begin{theorem}\label{SITARAM6}
The following are true for the associated semi-transition map $\varphi$.\\
$(a)$ $u \in A_{S}^{\times}$ if and only if $\varphi(u) = \emptyset$.\\
$(b)$ $\varphi(a) = \varphi(-a)$, for all $a \in A_{S}$.\\
$(c)$ $\varphi(a) \cap \varphi(b) \subseteq \varphi(ax + by)$, for all $a, b, x, y \in A_{S}$.\\
$(d)$ $(a, b) = A_{S}$, for some $a, b \in A_{S}$, if and only if $\varphi(a) \cap \varphi(b) = \emptyset$.
\end{theorem}
\begin{proof}
$(a)$. Assume first that $u \in A^{\times}$. Then there exists $v \in A$ such that $uv = 1$, and so $\varphi(u) \cup \varphi(v) = \emptyset$. Thus, $\varphi(u) = \emptyset$. Conversely, suppose that $\varphi(u) = \emptyset$. Let $u = \frac{j}{s}$, for some $j \in A$ and $s \in S$. Then $\psi(j) = \emptyset$, and so $j \in S$. This shows that $u \in A_{S}^{\times}$. \\
$(b)$. It is enough to show that $\psi(d) = \psi(-d)$, for all $d \in A$. So, let $d \in A$. Since $\psi(d) \cap \psi(-d) = \psi(d) \cap (\psi(d) \cup \psi(-1)) = \psi(d)$, and since $\psi(-1) = \emptyset$, so $\psi(d) \subseteq \psi(-d)$. Similarly, as $\psi(-d) \cap \psi(d) = \psi(-d) \cap (\psi(-d) \cup \psi(-1)) = \psi(-d)$, we see that $\psi(-d) \subseteq \psi(d)$. \\
$(c)$. It suffices to show that $\psi(c) \cap \psi(d) \subseteq \psi(cm + dn)$, for all $c, d, m, n \in A$. Since $\psi(cm) \cap \psi(dn) \subseteq \psi(cm + dn)$, and since $\psi(c) \subseteq \psi(cm)$ and $\psi(d) \subseteq \psi(dn)$, therefore the result follows. \\
$(d)$. If $(a, b) = A_{S}$, for some $a, b \in A_{S}$, then there exists $x, y \in A_{S}$ such that $ax + by = 1$. This implies that $\varphi(a) \cap \varphi(b) \subseteq \varphi(ax + by) = \varphi(1) = \emptyset$. Hence, we get that $\varphi(a) \cap \varphi(b) = \emptyset$. Conversely, assume that $\varphi(a) \cap \varphi(b) = \emptyset$. Since $\varphi(c) = \emptyset$, for some $c \in (a, b)$, we conclude that $c \in A_{S}^{\times}$. Thus, $(a, b) = A_{S}$. 
\end{proof}

\begin{definition}\label{SITARAM7}
A nonempty subfamily $\mathscr{F}$ of $\varphi(A_{S})$ is called a $\varphi$-filter on $\Lambda$ provided that\\
$(i)$ $\emptyset \notin \mathscr{F}$;\\
$(ii)$ if $\varphi(a), \varphi(b) \in \mathscr{F}$, then $\varphi(a) \cap \varphi(b) \in \mathscr{F}$; and\\
$(iii)$ if $\varphi(a) \in \mathscr{F}$, $\varphi(c) \in \varphi(A)$, and $\varphi(a) \subseteq \varphi(c)$, then $\varphi(c) \in \mathscr{F}$.
\end{definition}

\begin{theorem}\label{SITARAM8}
The following are true for the semi-transition map $\varphi$. \\
$(a)$ If $I$ is an ideal in $A_{S}$, then the family 
$$\varphi[I] = \{\varphi(a) \mid a \in I\}$$
is a $\varphi$-filter on $\Lambda$.\\
$(b)$ If $\mathscr{F}$ is an $\varphi$-filter on $\Lambda$, then the family 
$$\varphi^{-1}[\mathscr{F}] = \{a \in A_{S} \mid \varphi(a) \in \mathscr{F}\}$$
is an ideal in $A_{S}$.
\end{theorem}
\begin{proof}
$(a)$. $(i)$ Since $I$ contains no unit, $\emptyset \notin \varphi[I]$.\\
$(ii)$ Let $\varphi(a), \varphi(b) \in \varphi[I]$, where $a, b \in I$. Since $\varphi(a) \cap \varphi(b) = \varphi(c)$, for some $c \in (a, b)$, and since $(a, b) \subseteq I$, we conclude that $\varphi(a) \cap \varphi(b) = \varphi(c) \in \varphi[I]$.\\
$(iii)$ Let $\varphi(a) \in \varphi[I]$, where $a \in I$, and let $\varphi(a) \subseteq \varphi(r)$, for some $r \in A_{S}$. Since $ra \in I$, we get that $\varphi(r) = \varphi(r) \cup \varphi(a) = \varphi(ra) \in \varphi[I]$.\\
$(b)$. Let $J = \varphi^{-1}[\mathscr{F}]$. Since $\emptyset \notin \mathscr{F}$, $J$ contains no unit. Let $a, b \in J$ and let $r \in A_{S}$. Since $\varphi(a) \cap \varphi(b) \subseteq \varphi(a + b)$, and since $\varphi(b) = \varphi(-b)$, we get that $\varphi(a) \cap \varphi(b) \subseteq \varphi(a - b)$. Thus, $a - b \in J$. Since $\varphi(a) \subseteq \varphi(ra)$, and since $\varphi(a) \in \mathscr{F}$, we conclude that $\varphi(ra) \in \mathscr{F}$, that is, $ra \in J$.
\end{proof}

By a $\varphi$-ultrafilter on $\Lambda$ is meant a maximal $\varphi$-filter, that is, one not contained in any other $\varphi$-filter. Thus, an $\varphi$-ultrafilter is a maximal subfamily of $\varphi(A_{S})$ with the finite intersection property. It follows from the maximal principle that every subfamily of $\varphi(A_{S})$ with the finite intersection property is contained in some $\varphi$-ultrafilter.

\begin{theorem}\label{SITARAM9}
The following are true for the semi-transition map $\varphi$. \\
$(a)$ If $M$ is a maximal ideal in $A_{S}$, then $\varphi[M]$ is an $\varphi$-ultrafilter on $\Lambda$. \\
$(b)$ If $\mathscr{A}$ is an $\varphi$-ultrafilter on $\Lambda$, then $\varphi^{-1}[\mathscr{A}]$ is a maximal ideal in $A_{S}$. \\
The mapping $\varphi$ is one-one from the set of all maximal ideals in $A_{S}$ onto the set of all $\varphi$-ultrafilters.
\end{theorem}
\begin{proof}
Since $\varphi$ and $\varphi^{-1}$ preserve inclusion, the result follows immediately from Theorem \eqref{SITARAM8}.
\end{proof}

\begin{theorem}\label{SITARAM10}
The following are true for the semi-transition map $\varphi$. \\
$(a)$ Let $M$ be a maximal ideal in $A_{S}$. If $\varphi(a)$ meets every member of $\varphi[M]$, then $a \in M$. \\
$(b)$ Let $\mathscr{A}$ be an $\varphi$-ultrafilter on $\Lambda$. If $\varphi(a)$ meets every member of $\mathscr{A}$, then $\varphi(a) \in \mathscr{A}$.
\end{theorem}
\begin{proof}
By Theorem \eqref{SITARAM9}, the two statements are equivalent. Observe that in $(b)$, $\mathscr{A} \cup \{\varphi(a)\}$ generates an $\varphi$-filter. As this contains the maximal $\varphi$-filter $\mathscr{A}$, it must be $\mathscr{A}$.
\end{proof}

\begin{definition}\label{SITARAM11}
An ideal $I$ in $A_{S}$ is called an $\varphi$-ideal if $\varphi(a) \in \varphi[I]$ implies $a \in I$, that is, if $I = \varphi^{-1}[\varphi[I]]$.
\end{definition}

Note that every maximal ideal is an $\varphi$-ideal, and so, intersection of maximal ideals is an $\varphi$-ideal. Also, if $\mathscr{F}$ is an $\varphi$-filter, then $\varphi^{-1}[\mathscr{F}]$ is an $\varphi$-ideal. In particular, if $I$ is an ideal in $A_{S}$, then $\varphi^{-1}[\varphi[I]]$ is an $\varphi$-ideal. Moreover, every $\varphi$-ideal is semi-prime, and so, every $\varphi$-ideal is an intersection of prime ideals. 

\begin{theorem}\label{SITARAM12}
If $I^{2}$ is an $\varphi$-ideal for some ideal $I$ in $A_{S}$, then $I^{2} = I$.
\end{theorem}
\begin{proof}
Let $a \in I$. Then $a^{2} \in I^{2}$. This implies that $\varphi(a) = \varphi(a^{2}) \in \varphi[I^{2}]$. Since $I^{2}$ is an $\varphi$-ideal and since $\varphi(a) \in \varphi[I^{2}]$, we conclude that $a \in I^{2}$, that is, $I \subseteq I^{2}$.    
\end{proof}

An ideal $I$ in $A_{S}$ is called strongly irreducible if for ideals $J$ and $K$ in $A_{S}$, the inclusion $J \cap K \subseteq I$ implies that either $J \subseteq I$ or $K \subseteq I$. A slight generalization of this concept is semi-strongly irreducible ideal. An ideal $I$ in $A_{S}$ is called a semi-strongly irreducible if for ideals $J$ and $K$ in $A_{S}$, the inclusion $J \cap K \subseteq I$ implies that either $J^{2} \subseteq I$ or $K^{2} \subseteq I$. For further details see \cite{a08} and \cite{hy24}.

\begin{theorem}\label{SITARAM13}
The following statements are equivalent for a $\varphi$-ideal $I$ in $A_{S}$. \\
$(a)$ $I$ is semi-strongly irreducible.\\
$(b)$ $I$ is strongly irreducible.\\
$(c)$ $I$ is prime.
\end{theorem}
\begin{proof}
$(a) \Rightarrow (b)$. Assume that $J \cap K \subseteq I$. Then either $J^{2} \subseteq I$ or $K^{2} \subseteq I$. Without loss of generality, suppose that $J^{2} \subseteq I$. Let $a \in J$, then $a^{2} \in J^{2}$, and so $a^{2} \in I$. This implies that $\varphi(a) = \varphi(a^{2}) \in \varphi[I]$, and so $a \in I$. Hence, $J \subseteq I$.\\
$(b) \Rightarrow (c)$. Let $JK \subseteq I$. We now show that $J \cap K \subseteq I$. If $a \in (J \cap K) - I$, then $a^{2} \in JK$, and so $\varphi(a) = \varphi(a^{2}) \in \varphi[I]$. This implies that $a \in I$, a contradiction. Thus, $J \cap K \subseteq I$. As $I$ is strongly irreducible, so either $J \subseteq I$ or $K \subseteq I$.\\
$(c) \Rightarrow (a)$. If $J \cap K \subseteq I$, then $JK \subseteq I$, and so either $J \subseteq I$ or $K \subseteq I$, that is, either $J^{2} \subseteq I$ or $K^{2} \subseteq I$.
\end{proof}

By a prime $\varphi$-filter, we mean an $\varphi$-filter $\mathscr{F}$ that satisfies the following property: whenever $\varphi(a) \cup \varphi(b) \in \mathscr{F}$, then either $\varphi(a) \in \mathscr{F}$ or $\varphi(b) \in \mathscr{F}$.

\begin{theorem}\label{SITARAM14}
The following are true for the semi-transition map $\varphi$. \\
$(a)$ If $P$ is a prime $\varphi$-ideal in $A_{S}$, then $\varphi[P]$ is a prime $\varphi$-filter.\\
$(b)$ If $\mathscr{F}$ is a prime $\varphi$-filter on $\Lambda$, then $\varphi^{-1}[\mathscr{F}]$ is a prime $\varphi$-ideal.
\end{theorem}
\begin{proof}
$(a)$. Let $\varphi(a) \cup \varphi(b) \in \varphi[P]$, then $ab \in \varphi^{-1}[\varphi[P]] = P$. This implies that either $a \in P$ or $b \in P$, that is, either $\varphi(a) \in \varphi[P]$ or $\varphi(b) \in \varphi[P]$.\\
$(b)$. It is clear that $P = \varphi^{-1}[\mathscr{F}]$ is an $\varphi$-ideal. Suppose now that $ab \in P$. Then $\varphi(a) \cup \varphi(b) \in \mathscr{F}$. This implies that either $\varphi(a) \in \mathscr{F}$ or $\varphi(b) \in \mathscr{F}$, and so either $a \in P$ or $b \in P$.
\end{proof}

Recall that a ring for which there is a fixed integer $n \geq 2$ such that every element in the ring has an $n$-th root in the ring is called an $n$-th root ring.

\begin{theorem}\label{SITARAM15}
If $A_{S}$ is an $n$-th root ring, then $IJ = I \cap J$, for all $\varphi$-ideals $I, J$. 
\end{theorem}
\begin{proof}
Let $I, J$ be $\varphi$-ideals. Let $a \in I \cap J$. Then $b^{n} = a$, for some $b \in A$. Since $b^{n}$ is contained in both the ideals $I$ and $J$, therefore $\varphi(b) \in \varphi[I]$ and $\varphi(b) \in \varphi[J]$, that is, $b \in I$ and $b \in J$. As $b^{n - 1} \in I$ and $b \in J$, we obtain that $b^{n} = (b^{n - 1})(b) \in IJ$, that is, $a \in IJ$. This shows that $I \cap J \subseteq IJ$.
\end{proof}

\section{Structural representation of semiprimitive associated semi-transitional rings}\label{sec:structure}

Recall that a ring $B$ is said to be a semiprimitive ring if $\mathrm{J}(B) = 0$. For each $b \in B$, let  $\mathscr{M}(b) = \{M \in \mathrm{Max}(B) \mid b \in M\}$. For the associated semi-transitional ring $A_{S}$ attached to $A$, the $\varphi$-ultrafilters on $\Lambda$ is denoted by $\Omega_{\varphi}(\Lambda)$ and for each $a \in A_{S}$, let $\overline{\varphi(a)} = \{\mathscr{U} \in \Omega_{\varphi}(\Lambda) \mid \varphi(a) \in \mathscr{U}\}$.

\begin{theorem}\label{SITARAM16}
A semiprimitive ring $B$ is an associated semi-transitional ring if an only if $\mathscr{M}(a) \cap \mathscr{M}(b) = \mathscr{M}(c)$, for all $a, b \in B$ and for some $c \in (a, b)$.
\end{theorem}
\begin{proof}
At first, we show $(\Leftarrow)$ direction. Let $\Lambda = \mathrm{Max}(B)$. Define $\psi: B \rightarrow \powerset(\Lambda)$ by $\psi(b) = \mathscr{M}(b)$. It is easy to see that $\psi(ab) = \psi(a) \cup \psi(b)$, for all $a, b \in A$. Since no maximal ideal contains $1$, therefore $\psi(1) = \emptyset$. Further, $\mathscr{M}(a) = \Lambda$, if and only if, $a \in \bigcap_{M \in \Lambda} M$, if and only if, $a = 0$, since $\mathrm{J}(B) = 0$. Now, the hypotheses and the fact that $\mathscr{M}(a) \cap \mathscr{M}(b) \subseteq \mathscr{M}(a + b)$, for all $a, b \in A$, shows that $\psi(a) \cap \psi(b) = \psi(c) \subseteq \psi(a + b)$, for all $a, b \in A$ and for some $c \in (a, b)$. Finally, if $\mathscr{M}(a) \cap \mathscr{M}(b) = \emptyset$, then $(a, b)$ can not be contained in any maximal ideal of $B$, so $(a, b) = (1)$. Hence, we get that if $\psi(a) \cap \psi(b) = \emptyset$, then $(a, b)$ is principal. This shows that $B$ is a semi-transitional ring. Now, if we take $S = \{b \in B \mid \psi(b) = \emptyset\}$, then observe that $S = B^{\times}$. Since $B \cong B_{S}$, therefore we conclude that $B$ is an associated semi-transitional ring.

We now show $(\Rightarrow)$ direction. So, assume that $B := D_{S}$ is the associated semi-transitional ring attached to a semi-transitional ring $D$, where the semi-transition map $\varphi: D_{S} \rightarrow \powerset(\Lambda)$ is associated to the semi-transition map $\psi$ defined on $D$ and $S = \{d \in D \mid \psi(d) = \emptyset\}$. We now show that $\overline{\varphi(a)} \cap \overline{\varphi(b)} = \overline{\varphi(c)}$, for all $a, b \in A_{S}$ and for some $c \in (a, b)$ such that $\varphi(a) \cap \varphi(b) = \varphi(c)$. So, let $\mathscr{U} \in \overline{\varphi(a)} \cap \overline{\varphi(b)}$. This implies that $\varphi(a) \in \mathscr{U}$ and $\varphi(b) \in \mathscr{U}$. Since $\mathscr{U}$ is a $\varphi$-filter on $\Lambda$, therefore we have $\varphi(a) \cap \varphi(b) \in \mathscr{U}$. Since $\varphi(a) \cap \varphi(b) = \varphi(c)$, for some $c \in (a, b)$, so we get that $\varphi(c) \in \mathscr{U}$, that is, $\mathscr{U} \in \overline{\varphi(c)}$. This shows that $\overline{\varphi(a)} \cap \overline{\varphi(b)} \subseteq \overline{\varphi(c)}$. Conversely, suppose that $\mathscr{U} \in \overline{\varphi(c)}$. Then we have that $\varphi(a) \cap \varphi(b) = \varphi(c) \in \mathscr{U}$. Since, $\mathscr{U}$ is a $\varphi$-filter on $\Lambda$, we have that $\varphi(a) \in \mathscr{U}$ and $\varphi(b) \in \mathscr{U}$. Thus, $\mathscr{U} \in \overline{\varphi(a)} \cap \overline{\varphi(b)}$. This shows that $\overline{\varphi(c)} \subseteq \overline{\varphi(a)} \cap \overline{\varphi(b)}$. Observe now that $\varphi(a) \in \mathscr{U}$, if and only if, $a \in M$, where $M \in \mathrm{Max}(A_{S})$ is such that $M = \varphi^{-1}[\mathscr{U}]$. Now, since $\overline{\varphi(a)} \cap \overline{\varphi(b)} = \overline{\varphi(c)}$, for all $a, b \in A_{S}$ and for some $c \in (a, b)$, therefore we get that $\mathscr{M}(a) \cap \mathscr{M}(b) = \mathscr{M}(c)$, for all $a, b \in A_{S}$ and for some $c \in (a, b)$.
\end{proof}

Notice that if we are given any ring $B$, then the proof of Theorem \eqref{SITARAM16} in $(\Rightarrow)$ direction suggests that $B$ can not be an associated semi-transitional ring if $\mathscr{M}(a) \cap \mathscr{M}(b) \neq \mathscr{M}(c)$, for some $a, b \in B$ and for all $c \in (a, b)$. The following corollary is now an immediate consequence of Theorem \eqref{SITARAM16}.

\begin{corollary}\label{SITARAM17}
All semiprimitive B\'ezout rings are associated semi-transitional rings.
\end{corollary}
\begin{proof}
Let $B$ be a semiprimitive B\'ezout ring and let $\Lambda = \mathrm{Max}(B)$. We have already seen from the proof of Theorem \eqref{SITARAM16} that the map $\psi: B \rightarrow \powerset(\Lambda)$, defined by $\psi(b) = \mathscr{M}(b)$, satisfy the conditions $(i)$ and $(ii)$ of Definition \eqref{SITARAM}. Now, for any $a, b \in B$, we have that $(a, b) = (c)$, for some $c \in B$, since $B$ is a B\'ezout ring. Thus, we get that $\mathscr{M}(a) \cap \mathscr{M}(b) = \mathscr{M}(c)$, where $c \in (a, b)$ is such that $(a, b) = (c)$. As $\mathscr{M}(a) \cap \mathscr{M}(b) \subseteq \mathscr{M}(a + b)$, for all $a, b \in A$, and as $(a^{m}, b^{n})$ is principal, for all $a, b \in B$ and for all $m, n \in \mathbb{N}$, we see that $\psi$ satisfies the condition $(iii)$ of Definition \eqref{SITARAM}. This shows that $B$ is a semi-transitional ring. Now, observe that $S = \{b \in B \mid \psi(b) = \emptyset\} = B^{\times}$. Thus, we obtain that $B \cong B_{S}$. Hence, $B$ is an associated semi-transitional ring.
\end{proof}

\section{Transitional rings}\label{sec:transitional}

\begin{definition}\label{SITARAM18}
Let $\psi: A \rightarrow \powerset(\Lambda)$ be a semi-transition map. Then $A$ is said to be a transitional ring if $A$ is a pm-ring and $\psi$ satisfies the following:\\
$(i)$ if $\lambda_{1} \neq \lambda_{2}$ in $\Lambda$, then there exists $a \in A$ such that $\psi(a)$ contains exactly one of them; and \\
$(ii)$ for some $\lambda \in \Lambda$ and for some $a \in A$, if $\lambda \notin \psi(a)$, then there exists $b \in A$ such that $\lambda \in \psi(b)$ and $\psi(a) \cap \psi(b) = \emptyset$.\\
Any semi-transition map $\psi$ defined on a pm-ring $A$ with these two properties is called a transition map on $A$.
\end{definition}

Throughout this section, we assume $\psi: A \rightarrow \powerset(\Lambda)$ is a transition map and $A_{S} := S^{-1}A$, where $S= \{a \in A \mid \psi(a) = \emptyset\}$. In order to give examples of transitional rings, it is important to look at the following result.

\begin{theorem}\label{SITARAM19}
If $A$ is a transitional ring, then so is $A_{S}$.
\end{theorem}
\begin{proof}
By Theorem \eqref{SITARAM4} and \eqref{SITARAM4.5}, we know that $A_{S}$ is both a semi-transitional ring and a pm-ring. It remains to show that the map $\varphi$, introduced in Theorem \eqref{SITARAM4}, satisfy the conditions $(i)$ and $(ii)$ of Definition \eqref{SITARAM18}. But this is evident since $\psi(a) = \varphi(a)$, for all $a \in A$.
\end{proof}

Since it is clear from Theorem \eqref{SITARAM19} that $A_{S}$ is a transitional ring whenever $A$ is, therefore $A_{S}$ is called the associated transitional ring attached to $A$ and the map $\varphi$ is called the associated transition map. 

\begin{examples}
Consider the ring $C^{\ast}(X)$ introduced in Examples \eqref{SITARAM1}. Also, assume that $X$ is a Tychonoff space. As $C^{\ast}(X)$ is a semi-transitional pm-ring, therefore we need to verify only the conditions $(i)$ and $(ii)$ of Definition \eqref{SITARAM18}. But this is evident since the condition $(i)$ of Definition \eqref{SITARAM18} follows from the complete regularity of $X$ and the condition $(ii)$ of Definition \eqref{SITARAM18} follows from the fact that the sets $\{x \in X \mid f(x) \leq r\}$ and $\{x \in X \mid f(x) \geq r\}$ are zero-sets in $C^{\ast}(X)$ \cite{gj60}, where $f \in C^{\ast}(X)$ and $r \in \mathbb{R}$. Note that applying Theorem \eqref{SITARAM19}, we may conclude that the ring $C(X)$, mentioned in Examples \eqref{SITARAM5}, is also a transitional ring.
 
Now, consider the semi-transitional ring $C$ introduced in Examples \eqref{SITARAM2}. We already know from the demonstration in Examples \eqref{SITARAM5} that $C$ is a pm-ring. If $l, t \in \mathbb{N}$ be such that $l \neq t$, then for the sequence $\{x_{n}\}_{n = 1}^{\infty}$, defined by $x_{m} = 1$, for all $m \in \mathbb{N} - \{l\}$, and $x_{l} = 0$, we have $l \in \psi_{2}(\{x_{n}\}_{n = 1}^{\infty})$ and $t \notin \psi_{2}(\{x_{n}\}_{n = 1}^{\infty})$. This shows that $\psi_{2}$ satisfies the condition $(i)$ of Definition \eqref{SITARAM18}. Now, assume that for some $l \in \mathbb{N}$ and for some $\{x_{n}\}_{n = 1}^{\infty} \in C$, $l \notin \psi_{2}(\{x_{n}\}_{n = 1}^{\infty})$. If we consider the sequence $\{y_{n}\}_{n = 1}^{\infty}$, defined by $y_{m} = 1$, for all $m \in \mathbb{N} - \{l\}$ and $y_{l} = 0$, then observe that $l \in \psi_{2}(\{y_{n}\}_{n = 1}^{\infty})$ and $\psi_{2}(\{x_{n}\}_{n = 1}^{\infty}) \cap \psi_{2}(\{y_{n}\}_{n = 1}^{\infty}) = \emptyset$. This shows that $\psi_{2}$ satisfies the condition $(ii)$ of Definition \eqref{SITARAM18}. Further, note that applying Theorem \eqref{SITARAM19}, we may conclude that $C_{S}$ is also a transitional ring.  
\end{examples}

The question of whether there are infinitely many transitional rings is a natural one, just as it is with semi-transitional rings. The following result affirms the question.

\begin{theorem}\label{SITARAM20}
Let $\{A_{\alpha}\}_{\alpha \in F}$ be a nonempty indexed family of pm-rings and let $\psi_{\alpha}: A_{\alpha} \rightarrow \powerset(\Lambda_{\alpha})$ be a transition map for each $\alpha \in F$. Suppose there exists a fixed positive integer $m$ such that for every $\alpha \in F$, $\psi_{\alpha}(a_{\alpha}) \cap \psi_{\alpha}(b_{\alpha}) = \emptyset$ implies $(a_{\alpha}^{m}, b_{\alpha}^{m})$ is principal. Then the product ring $\prod_{\alpha \in F} A_{\alpha}$ is a transitional ring.
\end{theorem}
\begin{proof}
Theorem \eqref{SITARAM3} implies that $A = \prod_{\alpha \in F} A_{\alpha}$ is a semi-transitional ring. Also, theorem $(3.3)$ in \cite{c82} implies that $A$ is a pm-ring. Thus, it remains to show only that the map $\psi: A \rightarrow \powerset(\Lambda)$, defined in Theorem \eqref{SITARAM3}, satisfies the conditions $(i)$ and $(ii)$ of Definition \eqref{SITARAM18}.\\
$(i)$ Suppose that $\lambda_{1} \neq \lambda_{2}$, for some $\lambda_{1}, \lambda_{2} \in \Lambda$. Then we have the following two cases:

\textbf{Case 1:} There exists $\gamma \in F$ such that $\lambda_{1}, \lambda_{2} \in \Lambda_{\gamma}$.

Since $\lambda_{1} \neq \lambda_{2}$, there exists $a_{\gamma} \in A_{\gamma}$ such that $\psi_{\gamma}(a_{\gamma})$ contains exactly one of them. If we now consider the element $a = (a_{\alpha})_{\alpha \in F}$ such that $a_{\alpha} = 1$, for all $\alpha \in F - \{\gamma\}$, and $a_{\alpha} = a_{\gamma}$, for $\alpha = \gamma$, then $\psi(a)$ contains exactly one of them.

\textbf{Case 2:} There exists $\gamma, \delta \in F$ such that $\lambda_{1} \in \Lambda_{\gamma}$ and $\lambda_{2} \in \Lambda_{\delta}$.

In this case, if we consider the element $a = (a_{\alpha})_{\alpha \in F}$ such that $a_{\alpha} = 1$, for all $\alpha \in F - \{\gamma\}$, and $a_{\alpha} = 0$, for $\alpha = \gamma$, then $\psi(a)$ contains exactly one of them, namely $\lambda_{1}$.\\
$(ii)$ Let $a = (a_{\alpha})_{\alpha \in F}$ and suppose that $\lambda \notin \psi(a)$, for some $\lambda \in \Lambda$. Then there exists $\gamma \in F$ such that $\lambda \in \Lambda_{\gamma}$ and $\lambda \notin \psi_{\gamma}(a_{\gamma})$. Hence, there exists $b_{\gamma} \in A_{\gamma}$ such that $\lambda \in \psi_{\gamma}(b_{\gamma})$ and $\psi_{\gamma}(a_{\gamma}) \cap \psi_{\gamma}(b_{\gamma}) = \emptyset$. If we now consider the element $b = (b_{\alpha})_{\alpha \in F}$ in $A$ such that $b_{\alpha} = 1$, for all $\alpha \in F - \{\gamma\}$, and $b_{\alpha} = b_{\gamma}$, for $\alpha = \gamma$, then $\lambda \in \psi(b)$ and $\psi(a) \cap \psi(b) = \emptyset$.
\end{proof}

\section{An important class of transitional rings}\label{sec: ideal conv}

The notion of $I$-convergent sequences of reals was introduced by Kostyrko et al. \cite{k01} using the structure of the ideal $I$ of subsets of the set of natural numbers. An ideal $I$ of $\mathbb{N}$ is said to be a nontrivial ideal of $\mathbb{N}$ if $\mathbb{N} \notin I$. Throughout this section, we assume all ideals are nontrivial.

\begin{theorem}\label{SITARAM21}
Let $I$ be an ideal of $\mathbb{N}$ and let $A_{I}$ be the collection of all $I$-convergent sequences of real numbers. Then\\
$(a)$ $A_{I}$ is a commutative ring with unity.\\
$(b)$ An element $(x_{n})_{n \geq 1}$ in $A_{I}$ is a unit if and only if $x_{n} \neq 0$ for all $n \in \mathbb{N}$ and $I-\lim x_{n} \neq 0$.
\end{theorem}
\begin{proof}
$(a)$ It is well-known that if $I-\lim x_{n} = x$ and $I-\lim y_{n} = y$, then $I-\lim (x_{n} + cy_{n}) = x + cy$, for all $c \in \mathbb{R}$, and $I-\lim (x_{n}y_{n}) = xy$. Further, it is not hard to see that the additive identity of $A_{I}$ is the constant sequence $a_{n} = 0$ for all $n \in \mathbb{N}$ and the multiplicative identity of $A_{I}$ is the constant sequence $b_{n} = 1$ for all $n \in \mathbb{N}$. Finally, as distributive property trivially holds in $A_{I}$, the result follows. \\
$(b)$ Assume that $(x_{n})_{n \geq 1}$ is a unit in $A_{I}$. Then there exists $(y_{n})_{n \geq 1}$ in $A_{I}$ such that $(x_{n}y_{n})_{n \geq 1} = (b_{n})_{n \geq 1}$, where $b_{n} = 1$, for all $n \in \mathbb{N}$. But this implies $x_{n}y_{n} = 1$, for all $n \in \mathbb{N}$. Hence, $x_{n} \neq 0$ for all $n \in \mathbb{N}$. Finally, since we have $(I-\lim x_{n})(I-\lim y_{n}) = I-\lim (x_{n}y_{n}) = 1$, therefore we conclude that $I-\lim x_{n} \neq 0$.

Conversely, assume that $(x_{n})_{n \geq 1} \in A_{I}$ be such that $x_{n} \neq 0$ for all $n \in \mathbb{N}$ and $I-\lim x_{n} = x$, where $x \neq 0$. Let $\epsilon > 0$ be an arbitrary fixed real number and let $F(I)$ be the associated filter of $\mathbb{N}$ corresponding to the ideal $I$. Let us consider the following sets:
\begin{align*}
A_{1} = \{n \in \mathbb{N} \mid |x_{n} - x| \geq \frac{|x|}{2}\}, A_{2} = \{n \in \mathbb{N} \mid |x_{n} - x| \geq \frac{\epsilon|x|^{2}}{2}\},\\
A_{3} = \{n \in \mathbb{N} \mid ||x_{n}| - |x|| \geq \frac{|x|}{2}\}, \text{and } A_{4} = \{n \in \mathbb{N} \mid |\frac{1}{x_{n}} - \frac{1}{x}| \geq \epsilon\}.
\end{align*}
Since $A_{1}^{c} \subseteq A_{3}^{c}$ and since $A_{1}^{c} \in F(I)$, we conclude that $A_{3}^{c} \in F(I)$. Again, as $A_{2}^{c} \in F(I)$, we conclude that $A_{2}^{c} \cap A_{3}^{c} \in F(I)$. Let $n \in A_{2}^{c} \cap A_{3}^{c}$. Now, as $n \in A_{3}^{c}$, we conclude that $\frac{1}{|x_{n}||x|} < \frac{2}{|x|^{2}}$. Further, as $n \in A_{2}^{c}$, we get that $|x_{n} - x| < \frac{\epsilon|x|^{2}}{2}$. Thus, $\frac{|x_{n} - x|}{|x_{n}||x|} < \frac{2}{|x|^{2}} \cdot \frac{\epsilon|x|^{2}}{2}$, that is, $|\frac{1}{x_{n}} - \frac{1}{x}| < \epsilon$. Hence, we get that $n \in A_{2}^{c} \cap A_{3}^{c}$ implies $n \in A_{4}^{c}$, that is, $A_{2}^{c} \cap A_{3}^{c} \subseteq A_{4}^{c}$. As $A_{2}^{c} \cap A_{3}^{c} \in F(I)$, therefore we conclude that $A_{4}^{c} \in F(I)$, that is, $A_{4} \in I$. Hence, $I-\lim \frac{1}{x_{n}} = \frac{1}{x}$, and so $(\frac{1}{x_{n}})_{n \geq 1} \in A_{I}$. This proves that $(x_{n})_{n \geq 1}$ is a unit in $A_{I}$. 
\end{proof}

\begin{lemma}\label{SITARAM22}
Let $I$ be an ideal of $\mathbb{N}$ and let $r > 0$ be a fixed real number. Let $(x_{n})_{n \geq 1}$ be a sequence of reals and let $A(\epsilon) = \{n \in \mathbb{N} \mid |x_{n} - x| \geq \epsilon\}$. If $A(\epsilon) \in I$ for all $0 < \epsilon < r$, then $I-\lim x_{n} = x$.
\end{lemma}
\begin{proof}
Let $B(\eta) = \{n \in \mathbb{N} \mid |x_{n} - x| \geq \eta\}$, where $\eta \geq r$ be any fixed real number. Since $B(\eta) \subseteq A(\epsilon)$, for all $0 < \epsilon < r$, and since $A(\epsilon) \in I$, for all $0 < \epsilon < r$, we conclude that $B(\eta) \in I$, for all $\eta \geq r$, since $I$ is an ideal. If $\eta < r$, then from the hypotheses, we get that $B(\eta) \in I$. This implies that $B(\eta) \in I$, for all $\eta > 0$, and so $I-\lim x_{n} = x$.
\end{proof}

Recall that a nontrivial ideal $I$ of $\mathbb{N}$ is called an admissible ideal if $I$ contains the ideal $I_{f}$ of all finite subset of $\mathbb{N}$.

\begin{theorem}\label{SITARAM23}
If $I$ is an admissible ideal of $\mathbb{N}$, then $A_{I}$ is a transitional-ring.
\end{theorem}
\begin{proof}
Define a function $Z: A_{I} \rightarrow \powerset(\mathbb{N})$ by $Z((x_{n})_{n \geq 1}) = \{n \in \mathbb{N} \mid x_{n} = 0\}$. Let $N = \mathbb{N} \cup \{p\}$ with $p \notin \mathbb{N}$. Let $f: A_{I} \rightarrow \powerset(N)$ be a function defined by:
\begin{equation*}
f((x_{n})_{n \geq 1}) = \begin{cases} 
Z((x_{n})_{n \geq 1}), &\text{ if } I-\lim x_{n} \neq 0 \\
Z((x_{n})_{n \geq 1}) \cup \{p\}, &\text{ if } I-\lim x_{n} = 0
\end{cases}
\end{equation*}
At first we show that $(A_{I}, N, f)$ is a semi-transitional ring:\\
$(i)$ Let $a = (x_{n})_{n \geq 1} \in A_{I}$ and $b = (y_{n})_{n \geq 1} \in A_{I}$. Then $ab = (x_{n}y_{n})_{n \geq 1}$. It is evident that $Z(ab) = Z(a) \cup Z(b)$. Observe that we have $p \in f(a) \cup f(b)$ if and only if either $I-\lim x_{n} = 0$ or $I-\lim y_{n} = 0$, that is, if and only if $I-\lim (x_{n}y_{n}) = 0$, and this is true if and only if $p \in f(ab)$. This shows that $f(ab) = f(a) \cup f(b)$, for all $a, b \in A_{I}$.\\
$(ii)$ Let $a = (x_{n})_{n \geq 1} \in A_{I}$. If $x_{m} = 1$, for all $m \in \mathbb{N}$, then $I-\lim x_{n} = 1$, and so $f(a) = \emptyset$. Suppose now that $f(a) = N$. This implies that $x_{m} = 0$, for all $m \in \mathbb{N}$ and $I-\lim x_{n} = 0$. Thus, $a = 0$. Conversely, if $a = 0$, then $f(a) = N$. \\
$(iii)$ Let $a = (x_{n})_{n \geq 1} \in A_{I}$ and $b = (y_{n})_{n \geq 1} \in A_{I}$. It is not hard to see that $Z(a) \cap Z(b) = Z(a^{2} + b^{2})$. Observe now that we have $p \in f(a) \cap f(b)$ if and only if $I-\lim x_{n} = 0$ and $I-\lim y_{n} = 0$, and this is true if and only if $I-\lim (x_{n}^{2} + y_{n}^{2}) = 0$. If we take $c = (x_{n}^{2} + y_{n}^{2})_{n \geq 1}$, then it is now clear that $f(a) \cap f(b) = f(c) \subseteq f(a + b)$, where $c \in (a, b)$. If $f(a) \cap f(b) = \emptyset$, then applying Theorem \eqref{SITARAM21}$(b)$ we see that $a^{2} + b^{2} \in (a, b)$ is a unit.

By Theorem \eqref{SITARAM21}$(b)$, the condition $f(a) = \emptyset$ implies that $a$ is a unit in $A_{I}$, therefore we get that $A_{I} \cong (A_{I})_{S}$, where $S = \{a \in A_{I} : f(a) = \emptyset\}$, that is, $A_{I}$ is an associated transitional ring. Hence, throughout the proof we invoke Theorem \eqref{SITARAM6}$(d)$ and theorem $(4.1)$ in \cite{c82} to deduce that $A_{I}$ is a pm-ring. \\
Let $a = (x_{n})_{n \geq 1} \in A_{I}$ and $b = (y_{n})_{n \geq 1} \in A_{I}$ be such that $I-\lim x_{n} = x$ and $I-\lim y_{n} = y$. Let $r \in \mathbb{R}$ be a fixed real number. Assume now that $(a, b) = 1$. Then $f(a) \cap f(b) = \emptyset$. Hence, we have the following two cases and each case has the following sub-cases:
$$
\begin{tabular}{|c|c|c|}
\hline
\textbf{Case 1:} $x \neq 0$ \text{ and } $y \neq 0$ & \textbf{Case 2:} $x \neq 0$ \text{ and } $y = 0$ \\ \hline
\textbf{Sub-cases} &  \textbf{Sub-cases} \\ \hline
$(1)$ $Z(a) = \emptyset$ \text{ and } $Z(b) = \emptyset$ & $(1)$ $Z(a) = \emptyset$ \text{ and } $Z(b) = \emptyset$\\ \hline
$(2)$ $Z(a) = \emptyset$ \text{ and } $Z(b) \neq \emptyset$ & $(2)$ $Z(a) = \emptyset$ \text{ and } $Z(b) \neq \emptyset$ \\ \hline
$(3)$ $Z(a) \neq \emptyset$ \text{ and } $Z(b) = \emptyset$ & $(3)$ $Z(a) \neq \emptyset$ \text{ and } $Z(b) = \emptyset$ \\ \hline
$(4)$ $Z(a) \neq \emptyset$ \text{ and } $Z(b) \neq \emptyset$ & $(4)$ $Z(a) \neq \emptyset$ \text{ and } $Z(b) \neq \emptyset$ \\ \hline
\end{tabular}
$$ 
Observe that the sub-cases $(2)$ and $(3)$ in the case $(1)$ are symmetric, so we only prove the sub-cases $(1)$, $(2)$, and $(4)$ in the case $(1)$.

\textbf{Sub-case $(1)$ of Case $1$ and Sub-case $(1)$ of Case 2:} \\
Define sequences $c = (c_{n})_{n \geq 1}$ and $d = (d_{n})_{n \geq 1}$ as follows: $c_{1} = r$  and $c_{n} = 0$ for all $n \in \mathbb{N}$ such that $n \neq 1$; and $d_{1} = 0$ and $d_{n} = r$ for all $n \in \mathbb{N}$ such that $n \neq 1$. As $c$ and $d$ are ordinary convergent sequences converge to $0$ and $r$ respectively, and as $I$ is an admissible ideal, we get that $c, d \in A_{I}$. It is clear that $cd = 0$. Since $a$ is a unit, we get that $(a, c) = A_{I}$. Further, as $f(b) \cap f(d) = \emptyset$, therefore we get that $(b, d) = A_{I}$. \\
\textbf{Sub-case $(2)$ of Case $1$ and Sub-case $(2)$ of Case 2:} \\
Let $m \in \mathbb{N} - Z(b)$. Define sequences $c = (c_{n})_{n \geq 1}$ and $d = (d_{n})_{n \geq 1}$ as follows: $c_{m} = r$  and $c_{n} = 0$ for all $n \in \mathbb{N}$ such that $n \neq m$; and $d_{m} = 0$ and $d_{n} = r$ for all $n \in \mathbb{N}$ such that $n \neq m$. As $c$ and $d$ are ordinary convergent sequences converge to $0$ and $r$ respectively, and as $I$ is an admissible ideal, we get that $c, d \in A_{I}$. Since $a$ is a unit, so we get that $(a, c) = A_{I}$. Now, as $f(b) \cap f(d) = \emptyset$, therefore we have that $(b, d) = A_{I}$. \\
\textbf{Sub-case $(4)$ of Case $1$ and Sub-cases $(3), (4)$ of Case $2$:} \\
Define sequences $c = (c_{n})_{n \geq 1}$ and $d = (d_{n})_{n \geq 1}$ as follows: $c_{n} = 0$ if $x_{n} \neq 0$  and $c_{n} = r$ if $x_{n} = 0$; and $d_{n} = 0$ if $c_{n} \neq 0$ and $d_{n} = r$ if $c_{n} = 0$. Now we show that both $c$ and $d$ are elements of $A_{I}$. So suppose that $\epsilon > 0$ be any but fixed real number such that $\epsilon < \min\{|x|, |y|\}$. Let us now consider the sets
\begin{align*}
A_{1} = \{n \in \mathbb{N} \mid |x_{n} - x| \geq \epsilon\}, A_{2} = \{n \in \mathbb{N} \mid |y_{n} - y| \geq \epsilon\},\\
A_{3} = \{n \in \mathbb{N} \mid |c_{n}| \geq \epsilon\}, A_{4} = \{n \in \mathbb{N} \mid |d_{n} - r| \geq \epsilon\}.
\end{align*} 
Observe that if $n \in A_{1}^{c}$, then $x_{n} \neq 0$, and this implies that $c_{n} = 0$, that is, $n \in A_{3}^{c}$. This shows that $A_{1}^{c} \subseteq A_{3}^{c}$, that is, $A_{3} \subseteq A_{1}$. As $A_{1} \in I$, we get that $A_{3} \in I$, that is, $I-\lim c_{n} = 0$. Thus, $c \in A_{I}$. Similarly, observe that if $n \in A_{1}^{c}$, then $x_{n} \neq 0$, and this implies that $c_{n} = 0$, but this further implies that $d_{n} = r$, that is, $n \in A_{4}^{c}$. This shows that $A_{1}^{c} \subseteq A_{4}^{c}$, that is, $A_{4} \subseteq A_{1}$. As $A_{1} \in I$, we get that $A_{4} \in I$, that is, $I-\lim d_{n} = r$. Thus, $d \in A_{I}$.

It is clear that $cd = 0$ and $f(a) \cap f(c) = \emptyset$. As $f(b) \subseteq N - f(a) \subseteq f(c)$, we get that $f(b) \cap f(d) = \emptyset$, since $f(c) \cap f(d) = \emptyset$. Thus, $(a, c) = (b, d) = A_{I}$.

Finally, since $C \subseteq A_{I}$, where $C$ is the ring of real-valued convergent sequences of reals, it is not hard to conclude that $f$ satisfy the conditions $(i)$ and $(ii)$ of Definition \eqref{SITARAM18}.
\end{proof}

\begin{remark}\label{BM}
In Theorem \eqref{SITARAM23}, if we take $I$ as $I_{d}$ , the ideal consisting of all the subsets of $\mathbb{N}$ having natural density $0$, then Theorem \eqref{SITARAM23} implies that the ring $A_{I_{d}}$ of all statistical convergent sequences of reals is a transitional-ring. Observe that the ring $C$ of all convergent sequences of reals, which we proved to be a pm-ring in Examples \eqref{SITARAM5}, also follows directly from  Theorem \eqref{SITARAM23} by taking $I$ as $I_{f}$, the ideal consisting of all the finite subsets of $\mathbb{N}$. In this case, $C = A_{I_{f}}$. Now, the fact that $A_{I}$ is a pm-ring for any admissible ideal $I$, which we proved in Theorem \eqref{SITARAM23}, follows directly from corollary $(3)$ in \cite{m73}. Nevertheless, we provide direct proof here for completeness and to illustrate our method. 
\end{remark}

\section{\texorpdfstring{$\Omega_{f}$}{Omega}-compactification}\label{sec:omega}

Throughout this section, let $f: A \rightarrow \powerset(\Lambda)$ denote an associated transition map for some associated transitional ring $A$, where $A \cong D_{S}$ for some transitional ring $D$.

\begin{theorem}
Let $f: A \rightarrow \powerset(\Lambda)$ be an associated transition map. Then\\
$(a)$ the set $E_{\lambda} = \{f(a) \in f(A) \mid \lambda \in f(a)\}$ is an $f$-ultrafilter.\\
$(b)$ there is a one-to-one correspondence between $\Lambda$ and the set $\{E_{\lambda} \mid \lambda \in \Lambda\}$.
\end{theorem}
\begin{proof}
$(a)$. For some $b \in A$, assume that $f(b)$ meets every member of $E_{\lambda}$. If we show that $\lambda \in f(b)$, then we are done. So, assume on the contrary that $\lambda \notin f(b)$. Then there exists $c \in A$ such that $\lambda \in f(c)$ and $(b, c) = A$. Now, Theorem \eqref{SITARAM6}$(d)$ implies that $f(b) \cap f(c) = \emptyset$. As $f(c) \in E_{\lambda}$, we obtain a contradiction. \\
$(b)$. Define $\Psi: \Lambda \rightarrow \{E_{\lambda} \mid \lambda \in \Lambda\}$ by $\Psi(\lambda) = E_{\lambda}$. It is evident that $\Psi$ is surjective. Assume now that $\lambda_{1} \neq \lambda_{2}$. Then there exists $a \in A$ such that $f(a)$ contains exactly one of them. Without loss of generality, assume that $\lambda_{1} \in f(a)$ and $\lambda_{2} \notin f(a)$. This implies that $f(a) \in E_{\lambda_{1}}$ and $f(a) \notin E_{\lambda_{2}}$, that is, $E_{\lambda_{1}} \neq E_{\lambda_{2}}$.
\end{proof}

\begin{remark}
Since the sets $\Lambda$ and $\{E_{\lambda} \mid \lambda \in \Lambda\}$ are equipotent, from now on we identify $E_{\lambda}$ with $\lambda$.
\end{remark}

\begin{theorem}
Let $f$ be an associated transition map on $A$. A necessary and sufficient condition for an ideal $I$ in $A$ to be an $f$-ideal is the following: given $a \in A$, if there exists $b \in I$ such that $a$ belongs to every maximal ideal containing $b$, then $a \in I$, that is, $\mathscr{M}(b) \subseteq \mathscr{M}(a)$ and $b \in I \Rightarrow a \in I$.
\end{theorem}
\begin{proof}
At first assume that $I$ is an $f$-ideal. Also, suppose that $a \in A$ is such that $a$ belongs to every maximal ideal containing $b$ for some $b \in I$. We show first that $f(b) \subseteq f(a)$. Let $\lambda \in f(b)$. Then $f(b) \in E_{\lambda}$. This implies that $b \in f^{-1}[E_{\lambda}]$. As $E_{\lambda}$ is an $f$-ultrafilter, therefore $f^{-1}[E_{\lambda}]$ is a maximal ideal containing $b$, and so $a \in f^{-1}[E_{\lambda}]$. This shows that $f(a) \in E_{\lambda}$, that is, $\lambda \in f(a)$. Observe now that since $f(b) \in f[I]$, and since $f[I]$ is an $f$-filter, therefore we get that $f(a) \in f[I]$, that is, $a \in I$.

Conversely, assume that $f(c) \in f[I]$ for some $c \in A$. Then $f(c) = f(b)$ for some $b \in I$. Let $M$ be any maximal ideal containing $b$. Then $f(b) \in f[M]$, and so $f(c) \in f[M]$. As $M$ is an $f$-ideal, therefore we get that $c \in M$. Thus, we obtain that $c \in I$.
\end{proof}

Recall that $\Omega_{f}(\Lambda)$ denotes the collection of all $f$-ultrafilters on $\Lambda$ and for all $a \in A$, $\overline{f(a)} = \{\mathscr{U} \in \Omega_{f}(\Lambda) \mid f(a) \in \mathscr{U}\}$, as introduced in Section \eqref{sec:structure}.

\begin{theorem}
Let $f: A \rightarrow \powerset(\lambda)$ be an associated transition map. Then \\
$(a)$ The collection $\{f(a) \mid a \in A\}$ forms a closed base for some topology on $\Lambda$. \\
$(b)$ $f(a) \subseteq \overline{f(a)}$, for all $a \in A$. \\
$(c)$ $\overline{\Lambda} = \Omega_{f}(\Lambda)$. \\
$(d)$ The collection $\{\overline{f(a)} \mid a \in A\}$ forms a closed base for some topology on $\Omega_{f}(\Lambda)$. \\
$(e)$ $\overline{f(a)} \cap \Lambda = f(a)$, for all $a \in A$. \\
$(f)$ $\overline{f(a)} = Cl_{\Omega_{f}(\Lambda)}(f(a))$, for all $a \in A$. \\
$(g)$ $\overline{f(a)} \cup \overline{f(b)} = \overline{f(ab)}$, for all $a, b \in A$. 
\end{theorem}
\begin{proof}
$(a)$. Since $\cap_{a \in A} f(a) = \emptyset$, and since $f(a) \cap f(b) = f(c)$, for all $a, b \in A$ and for some $c \in (a, b)$, the result follows. \\
$(b)$. Let $\lambda \in f(a)$. Then $f(a) \in E_{\lambda}$ and so $E_{\lambda} \in \overline{f(a)}$, that is, $\lambda \in \overline{f(a)}$. \\
$(c)$. Since every $f$-ultrafilter in $\Omega_{f}(\Lambda)$ contains the set $\Lambda$, therefore $\overline{\Lambda} = \Omega_{f}(\Lambda)$. \\
$(d)$. Observe that $p \in \overline{f(a)} \cap \overline{f(b)}$ if and only if $p \in \overline{f(a)}$ and $p \in \overline{f(b)}$, that is, if and only if $f(a) \in p$ and $f(b) \in p$, and this is true if and only if $f(a) \cap f(b) \in p$, since $p$ is closed under superset. As $f(a) \cap f(b) = f(c)$ for some $c \in A$, therefore we get that $f(a) \cap f(b) \in p$ if and only if $p \in \overline{f(c)} = \overline{f(a) \cap f(b)} \in p$. Also, $\cap_{a \in A} \overline{f(a)} = \emptyset$. Hence, $\{\overline{f(a)} \mid a \in A\}$ forms a closed base for some topology on $\Omega_{f}(\Lambda)$. \\
$(e)$. Clearly, $f(a) \subseteq \overline{f(a)} \cap \Lambda$.
Suppose now that $\lambda \in \overline{f(a)} \cap \Lambda$. Then $f(a) \in E_{\lambda}$, and so $\lambda \in f(a)$. Thus, $\overline{f(a)} \cap \Lambda = f(a)$. \\
$(f)$. As $f(a) \subseteq \overline{f(a)}$, so $Cl_{\Omega_{f}(\Lambda)}(f(a)) \subseteq \overline{f(a)}$. Suppose now that $p \in \overline{f(a)}$. Let $\Omega_{f}(\Lambda) - \overline{f(b)}$ be a basic open set containing $p$. As $f(b) \notin p$ and as $p$ is an $f$-ultrafilter, therefore there exists $c \in A$ such that $f(c) \in p$ and $f(c) \cap f(b) = \emptyset$. As $f(a) \cap f(c) \in p$, let $\lambda \in f(a) \cap f(c)$. Clearly, $\lambda \notin f(b)$. This implies that $\lambda \notin \overline{f(b)}$, and so $\lambda \in \Omega_{f}(\Lambda) - \overline{f(b)}$. As $\lambda \in f(a)$, therefore $(\Omega_{f}(\Lambda) - \overline{f(b)}) \cap f(a) \neq \emptyset$. Thus, $p \in Cl_{\Omega_{f}(\Lambda)}(f(a))$. \\
$(g)$. If $p \in \overline{f(a)} \cup \overline{f(b)}$, then either $f(a) \in p$ or $f(b) \in p$. This implies that $f(ab) = f(a) \cup f(b) \in p$, and so $p \in \overline{f(ab)}$. Thus, $\overline{f(a)} \cup \overline{f(b)} \subseteq \overline{f(ab)}$. Now, if $p \in \overline{f(ab)}$, then $f(a) \cup f(b) \in p$. As $p$ is a prime $f$-filter, therefore either $f(a) \in p$ or $f(b) \in p$, that is, $p \in \overline{f(a)} \cup \overline{f(b)}$. Hence, $\overline{f(ab)} \subseteq \overline{f(a)} \cup \overline{f(b)}$.
\end{proof}

\begin{theorem}\label{prabhu}
Let $f: A \rightarrow \powerset(\Lambda)$ be an associated transition map. Then\\
$(a)$ the space $\Omega_{f}(\Lambda)$ is a compact Hausdorff space. \\
$(b)$ $\mathrm{Max}(A)$ is homeomorphic to $\Omega_{f}(\Lambda)$.
\end{theorem}
\begin{proof}
$(a)$. Let $\{\overline{f(a)} \mid a \in F\}$ be a family of basic closed sets having the finite intersection property. Let $a_{1}, a_{2}, \ldots, a_{n}$ be any finite number of elements of $F$. Since $\emptyset \neq \overline{f(a_{1})} \cap \overline{f(a_{2})} \cap \cdots \cap \overline{f(a_{n})} = \overline{f(a_{1}) \cap f(a_{2}) \cap \cdots \cap f(a_{n})}$, therefore $f(a_{1}) \cap f(a_{2}) \cap \cdots \cap f(a_{n}) \neq \emptyset$ any finite number of elements of $F$. This implies that the family $S = \{f(a) \mid a \in F\}$ also has the finite intersection property. Hence, $S$ can be extended to an $f$-ultrafilter $p$ on $\Lambda$. This implies that $f(a) \in p$ for all $a \in F$, that is, $p \in \overline{f(a)}$ for all $a \in F$. Thus, $p \in \bigcap_{a \in F} \overline{f(a)} \neq \emptyset$.

We now show that $\Omega_{f}(\Lambda)$ is a Hausdorff space. Let $p$ and $q$ be two distinct $f$-ultrafilters on $\Lambda$. As $p \neq q$, there exists $f(a) \in p$ and $f(b) \in q$ such that $f(a) \cap f(b) = \emptyset$. Thus, there exists $c \in Ann(d)$ such that $(a, c) = (b, d) = A$, that is, $f(c) \cup f(d) = \Lambda$ and $f(a) \cap f(c) = f(b) \cap f(d) = \emptyset$. As $f(c) \notin p$, we obtain that $p \notin \overline{f(c)}$, that is, $p \in \Omega_{f}(\Lambda) - \overline{f(c)}$. Similarly, as $f(d) \notin q$, we obtain that $q \in \Omega_{f}(\Lambda) - \overline{f(d)}$. As $(\Omega_{f}(\Lambda) - \overline{f(c)}) \cap (\Omega_{f}(\Lambda) - \overline{f(d)}) = \Omega({\Lambda}) - (\overline{f(c)} \cup \overline{f(d)})$, and as $\overline{f(c)} \cup \overline{f(d)} = \overline{\Lambda}$, we get that $(\Omega_{f}(\Lambda) - \overline{f(c)}) \cap (\Omega_{f}(\Lambda) - \overline{f(d)}) = \emptyset$.\\
$(b)$. Define $\varphi: \mathrm{Max}(A) \rightarrow \Omega_{f}(\Lambda)$ by $\varphi(M) = f(M)$. It is evident that $\varphi$ is a bijection. Let $\mathscr{M}(a) = \{M \in \mathrm{Max}(A) \mid a \in M\}$ be a basic closed set in $\mathrm{Max}(A)$. As $\varphi(\mathscr{M}(a)) = \overline{f(a)}$, we conclude that $\varphi$ is a homeomorphism. 
\end{proof}

\begin{example}\label{nkb89}
Let $\mathbb{N}^{\ast}$ be the one point compactification of $\mathbb{N}$. Consider the associated transition-map $f: C \rightarrow \powerset(N)$, where $C$ is the ring of all convergent sequences of reals and $N = \mathbb{N} \cup \{p\}$ with $p \notin \mathbb{N}$, defined by 
$$
f(g) = 
\begin{cases} Z(g) & \text{if } \lim_{n \to \infty} g(n) \ne 0 \\ Z(g) \cup \{p\} & \text{if } \lim_{n \to \infty} g(n) = 0 
\end{cases}
$$
As we already mentioned in Remark \eqref{BM}, this is the special case of Theorem \eqref{SITARAM23}, since $C = A_{I_{f}}$.  Let $k$ be a fixed natural number and let $g = (x_{n})_{n \geq 1}$ be a sequence in $C$ such that $x_{m} = 0$, for all $m \in \mathbb{N} - \{k\}$, and $x_{k} = 1$. As $N - f(g) = \{k\}$, we see that $\{k\}$ is open, for all $k \in \mathbb{N}$. If $\{p\}$ is open, then there exists $(y_{n})_{n \geq 1} = h \in C$ such that $p \in N - f(h) \subseteq \{p\}$, that is, $N - f(h) = \{p\}$. This implies that $f(h) = \mathbb{N}$, that is, $\lim_{n \to \infty} y_{n} = 0$. But this implies that $f(h) = N$, a contradiction. This shows that $\{p\}$ is not open, and so $p$ is the unique limit point of $N$. Suppose now that $U$ be an arbitrary open set containing the point $p$. Then there exists $(z_{n})_{n \geq 1} = t \in C$ such that $p \in N - f(t) \subseteq U$, and so $\lim_{n \to \infty} z_{n} \neq 0$. But this implies that $Z(t)$ is a finite set, and so $f(t)$ is a finite set. This shows that every neighborhood of $p$ is co-finite. Thus, $N$ is compact, and so $\Omega_{f}(N) \cong N \cong \mathbb{N}^{\ast}$.
\end{example}

\begin{theorem}\label{ramramram}
Let $f: A \rightarrow \powerset(\Lambda)$ and $g: B \rightarrow \powerset(K)$ be associated transition maps such that $K$ is compact with respect to the topology which is generated by the closed base $\{g(b) \mid b \in B\}$. For any continuous function $\sigma: \Lambda \rightarrow K$, if there exists a ring homomorphism $\varphi: B \rightarrow A$ such that the first diagram in \eqref{star} commutes, where $\sigma^{\ast}$ is defined by $S \mapsto \sigma^{-1}(S)$ for all $S \in g(B)$, then $\sigma$ has a unique continuous extension to $\Omega_{f}(\Lambda)$, that is, there exists a unique continuous function $\overline{\sigma}: \Omega_{f}(\Lambda) \rightarrow K$ such that the second diagram in \eqref{star} also commutes.
\begin{equation}\label{star}
\begin{tikzcd}
B && {g(B)} &&& {\Omega_{f}(\Lambda)} \\
\\
A && {f(A)} && \Lambda && K
\arrow["g", from=1-1, to=1-3]
\arrow["\varphi"', from=1-1, to=3-1]
\arrow["{\sigma^{\ast}}", from=1-3, to=3-3]
\arrow["{\overline{\sigma}}", dashed, from=1-6, to=3-7]
\arrow["f"', from=3-1, to=3-3]
\arrow["i", from=3-5, to=1-6]
\arrow["\sigma"', from=3-5, to=3-7]
\end{tikzcd}
\end{equation}
\end{theorem}
\begin{proof}
For each $p \in \Omega_{f}(\Lambda)$, define $q_{p} = \{g(b) \in g(B) \mid f(\varphi(b)) = \sigma^{-1}(g(b)) \in p\}$. At first, we show that $q_{p}$ is a prime $g$-filter:\\
$(i)$ Since $\sigma^{-1}(\emptyset) = \emptyset$ and since $\emptyset \notin p$, therefore we conclude that $\emptyset \notin q_{p}$.\\
$(ii)$ Let $g(b_{1}), g(b_{2}) \in q_{p}$ and let $g(c) = g(b_{1}) \cap g(b_{2})$ for some $c \in (b_{1}, b_{2})$. Then $\sigma^{-1}(g(c)) = \sigma^{-1}(g(b_{1})) \cap \sigma^{-1}(g(b_{2})) \in p$, and so $g(c) = g(b_{1}) \cap g(b_{2}) \in q_{p}$.\\
$(iii)$ Assume now that $g(b) \subseteq g(d)$ for some $d \in B$. Then $\sigma^{-1}(g(b)) \subseteq \sigma^{-1}(g(d))$. As $\sigma^{-1}(g(b)) \in p$, we conclude that $\sigma^{-1}(g(d)) \in p$, that is, $g(d) \in q_{p}$. 

Assume that $g(b_{1}) \cup g(b_{2}) \in q_{p}$. As $\sigma^{-1}(g(b_{1}b_{1})) = \sigma^{-1}(g(b_{1})) \cup \sigma^{-1}(g(b_{2}))$, and as $p$ is a prime $f$-filter on $\Lambda$, therefore we conclude that either $\sigma^{-1}(g(b_{1})) \in p$ or $\sigma^{-1}(g(b_{2})) \in p$, that is, either $g(b_{1}) \in q_{p}$ or $g(b_{2}) \in q_{p}$.

Now we show that $M = \bigcap_{g(b) \in q} g(b)$ consists of a single point, denoted by $k_{p}$:\\
On the contrary, assume that $k_{1}, k_{2} \in M$ for some $k_{1}, k_{2} \in K$ with $k_{1} \neq k_{2}$. Then there exists $b_{1} \in B$ such that $g(b_{1})$ contains exactly one of them. Without loss of generality, assume that $k_{1} \in g(b_{1})$. Since $k_{2} \notin g(b_{1})$, there exists $b_{2} \in B$ such that $k_{2} \in g(b_{2})$ and $g(b_{1}) \cap g(b_{2}) = \emptyset$. This implies that there exists $b_{3} \in Ann(b_{4})$ such that $(b_{1}, b_{3}) = (b_{2}, b_{4}) = B$. Hence, we get that $g(b_{3}) \cup g(b_{4}) = K$, and $g(b_{1}) \cap g(b_{3}) = g(b_{2}) \cap g(b_{4}) = \emptyset$. As $K \in q$, and as $q$ is a prime $g$-filter, therefore we obtain that either $g(b_{3}) \in q$ or $g(b_{4}) \in q$. If $g(b_{3}) \in q$, then $k_{1} \notin g(b_{3})$ since $g(b_{1}) \cap g(b_{3}) = \emptyset$ and $k_{1} \in g(b_{1})$. But this is a contradiction since $k_{1} \in M$. Similarly, if $g(b_{4}) \in q$, then $k_{2} \notin g(b_{4})$ since $g(b_{2}) \cap g(b_{4}) = \emptyset$ and $k_{2} \in g(b_{2})$, a contradiction, since $k_{2} \in M$.

Now, we define a map $\overline{\sigma}: \Omega_{f}(\Lambda) \rightarrow K$ by $\overline{\sigma}(p) = k_{p}$. At first we show that $\overline{\sigma}|_{\Lambda} = \sigma$. Let $p = E_{\lambda}$. Now, $q_{p} = \{g(b) \in g(B) \mid f(\varphi(b)) = \sigma^{-1}(g(b)) \in E_{\lambda}\}$. This implies that $q = \{g(b) \in g(B) \mid \sigma(\lambda) \in g(b)\}$. As we have already seen that $M = \bigcap_{g(b) \in q} g(b)$ is a singleton set, and since $\sigma(\lambda) \in g(b)$ for all $g(b) \in q$, therefore we obtain that $k_{p} = \sigma(\lambda)$, that is, $\overline{\sigma}(\lambda) = \sigma(\lambda)$.

Finally, we show that $\overline{\sigma}$ is a continuous map. Let $p \in \Omega_{f}(\Lambda)$ and let $K - g(b_{1})$ be a basic open set containing $\overline{\sigma}(p) = k_{p}$. As $k_{p} \notin g(b_{1})$, there exists $b_{2} \in B$ such that $k_{p} \in g(b_{2})$ and $g(b_{1}) \cap g(b_{2}) = \emptyset$. Hence, for some $b_{3} \in Ann(b_{4})$, we have that $g(b_{1}) \cap g(b_{3}) = g(b_{2}) \cap g(b_{4}) = \emptyset$ and $g(b_{3}) \cup g(b_{4}) = K$. As $k_{p} \notin g(b_{4})$, we get that $g(b_{4}) \notin q_{p}$, that is, $\sigma^{-1}(g(b_{4})) \notin p$, and so $p \in N = \Omega_{f}(\Lambda) - \overline{\sigma^{-1}(g(b_{4}))}$. We now show that $\overline{\sigma}(N) \subseteq K - g(b_{1})$. Let $t \in N$. Then $\sigma^{-1}(g(b_{4})) \notin t$, and this implies that $g(b_{4}) \notin q_{t}$, and so $k_{t} \notin g(b_{4})$. This implies that $k_{t} \in g(b_{3})$, and so $k_{t} \in K - g(b_{1})$. This shows that $\overline{\sigma}(t) = k_{t} \in K - g(b_{1})$.
\end{proof}

\section{Application}\label{sec:application}

\textbf{(I)} Consider the associated transition map $f: A_{I} \rightarrow \powerset(N)$ introduced in the proof of Theorem \eqref{SITARAM23}. Example \eqref{nkb89} shows that $N$ is a compact Hausdorff space, if $I = I_{f}$, namely $N \cong \mathbb{N}^{\ast}$.

If $I$ properly contains $I_{f}$, then $N$ need not be compact. For instance, consider the ideal $I$ generated by the set $S$ of all even natural numbers and $I_{f}$. It is easy to see that $\mathbb{N} \notin I$. Now, define a sequence $g = (x_{n})_{n \geq 1}$ as follows:
$$
x_{n} = 
\begin{cases} 0 & \text{if } n \in S \\
1 & \text{if } n \notin S
\end{cases}
$$
We claim that $I-\lim x_{n} = 1$. To prove this, observe that if $0 < \epsilon \leq 1$, then $\{n \mid |x_{n} - 1| \geq \epsilon\} = S \in I$. If $1 < \epsilon$, then $\{n \mid |x_{n} - 1| \geq \epsilon\} = \emptyset \in I$. Since $C \subseteq A_{I}$, Example \eqref{nkb89} implies that $\{k\}$ is open, for all $k \in \mathbb{N}$. Observe that $W_{p} = N - f(g)$ is an open set containing $p$. For $k \in \mathbb{N}$, if $W_{k} = \{2k\}$, then $\mathscr{F} = \{W_{m}\}_{m \in N}$ is an open cover for $N$. It is now evident that no finite sub-cover of $\mathscr{F}$ can cover $N$. This shows that $N$ is not compact.  

Because $N$ is not necessarily compact, we now turn to a general examination of the compactification $\Omega_{f}(N)$ of $N$ and investigate its extension property. 

So, let $K$ be a compact Hausdorff space and let $\sigma: N \rightarrow K$ be any continuous function. If $\sigma(N)$ is finite, then as $K$ is Hausdorff, there exists an open set $U$ of $\sigma(p)$ in $K$ such that $\sigma(N) \cap U = \{\sigma(p)\}$. This implies that $N \cap \sigma^{-1}(U) = \{p\}$, a contradiction, since $p$ is the limit point of $N$. This shows that $\sigma(N)$ is infinite.

We now show that if $h \in C(K)$, then $I-\lim h(\sigma(n)) = h(\sigma(p))$. Let $\epsilon > 0$ be any. Consider the neighborhood $W = [h(\sigma(p)) - \epsilon, h(\sigma(p)) + \epsilon]$ of $h(\sigma(p))$. Since, $U = \sigma^{-1}(h^{-1}(W))$ is a neighborhood of $p$, there exists $g \in A_{I}$ such that $p \in N - f(g) \subseteq U$. This implies, from the definition of $f$, that $I-\lim g(n) = l$ and $l \neq 0$. Choose $\delta < \mathrm{min}\{\epsilon, l\}$. As $f(g) \subseteq \{n \in \mathbb{N} \mid |g(n) - l| \geq \delta\} \in I$, we see that $f(g) \in I$, and so $N - U \in I$. Now, for all $n \in N - U$, $h(\sigma(n)) \notin W$. This shows that $N - U = \{n \in \mathbb{N} \mid |h(\sigma(n)) - h(\sigma(p))| \geq \epsilon\} \in I$. This proves that $\{h(\sigma(n))\}_{n \geq 1} \in A_{I}$, for all $h \in C(K)$.

Now, we define the associated transition maps $Z: C(K) \rightarrow \powerset(K)$ given by:
\begin{align*}
Z(\alpha) = \{x \in K : \alpha(x) = 0\}.
\end{align*}
As $K$ is a completely regular, the topology on $K$ generated by the closed base $\{Z_{1}(\alpha) \mid \alpha \in C(K)\}$ is equal to the original topology on $K$. Hence, $K$ is compact with respect to the topology which is generated by the closed base $\{Z_{1}(\alpha) \mid \alpha \in C(K)\}$. As we already obtained that $\{h(\sigma(n))\}_{n \geq 1} \in A_{I}$, for all $h \in C(K)$, therefore we define a map $\varphi: C(K) \rightarrow A_{I}$ by $\varphi(h) = \{h(\sigma(n))\}_{n \geq 1}$, for all $h \in C(K)$. It is easy to see that $\varphi$ is a ring homomorphism. Suppose now that $h \in C(K)$ be any. If $h(\sigma(p)) = 0$, then $\sigma(p) \in Z_{1}(h)$, and so $p \in \sigma^{-1}(Z_{1}(h))$. As $I-\lim h(\sigma(n)) = 0$, we see that $f(\varphi(h)) = Z(h \circ \sigma) \cup \{p\}$. Hence, $\sigma^{-1}(Z_{1}(h)) = f(\varphi(h))$, that is, the first diagram in \eqref{MAHARAJJIJH} commutes. If $h(\sigma(p)) \neq 0$, then also the first diagram in \eqref{MAHARAJJIJH} commutes. Therefore, applying Theorem \eqref{ramramram}, we conclude that there is a unique continuous function $\overline{\sigma}: \Omega_{f}(N) \rightarrow K$, such that $\overline{\sigma}|_{N} = \sigma$, that is, the second diagram in \eqref{MAHARAJJIJH} also commutes.  We can now apply theorems $(6.4)$ and $(6.7)$ in \cite{gj60} to conclude that $\Omega_{f}(N) \cong \beta N$. 
\begin{equation}\label{MAHARAJJIJH}
\begin{tikzcd}
{C(K)} && {Z_{1}(C(K))} &&& {\Omega_{f}(N)} \\
\\
{A_{I}} && {f(A_{I})} && N && K
\arrow["Z_{1}", from=1-1, to=1-3]
\arrow["\varphi"', from=1-1, to=3-1]
\arrow["{\sigma^{\ast}}", from=1-3, to=3-3]
\arrow["{\overline{\sigma}}", dashed, from=1-6, to=3-7]
\arrow["f"', from=3-1, to=3-3]
\arrow["i", from=3-5, to=1-6]
\arrow["\sigma"', from=3-5, to=3-7]
\end{tikzcd}
\end{equation}

\vspace{0.2cm}

\textbf{(II)} Let $X$ be a Tychonoff space and $K$ be a compact Hausdorff space. Define the associated transition maps $Z_{1}: C(K) \rightarrow \powerset(K)$ and $Z_{2}: C(X) \rightarrow \powerset(X)$ by:
\begin{align*}
Z_{1}(\alpha) = \{x \in K : \alpha(x) = 0\} \text{ and } Z_{2}(\beta) = \{x \in X : \beta(x) = 0\}.
\end{align*}
Since $K$ is completely regular, therefore the topology generated by the closed base $\{Z_{1}(\alpha) \mid \alpha \in C(K)\}$ is equal to the original topology on $K$. Thus, $K$ is compact with respect to the topology which is generated by the closed base $\{Z_{1}(\alpha) \mid \alpha \in C(K)\}$. Let $\sigma: X \rightarrow K$ be any continuous function. Then there is a ring homomorphism $\varphi: C(K) \rightarrow C(X)$, defined by $\varphi(\alpha) = \alpha \circ \sigma$, for all $\alpha$ in $C(K)$. Now, observe that we have $\sigma^{-1}(Z_{1}(\alpha)) = Z_{2}(\alpha \circ \sigma) = Z_{2}(\varphi(\alpha))$. Thus, we obtain that the first diagram in \eqref{MAHARAJJI} commutes. Hence, applying Theorem \eqref{ramramram}, we conclude that $\sigma$ has a continuous extension $\overline{\sigma}: \Omega_{f}(X) \rightarrow K$ such that the second diagram in \eqref{MAHARAJJI} also commutes. We can now apply theorems $(6.4)$ and $(6.7)$ in \cite{gj60} to conclude that $\Omega_{f}(X) \cong \beta X$. 
\begin{equation}\label{MAHARAJJI}
\begin{tikzcd}
{C(K)} && {Z_{1}(C(K))} &&& {\Omega_{f}(X)} \\
\\
{C(X)} && {Z_{2}(C(X))} && X && K
\arrow["Z_{1}", from=1-1, to=1-3]
\arrow["\varphi"', from=1-1, to=3-1]
\arrow["{\sigma^{\ast}}", from=1-3, to=3-3]
\arrow["{\overline{\sigma}}", dashed, from=1-6, to=3-7]
\arrow["Z_{2}"', from=3-1, to=3-3]
\arrow["i", from=3-5, to=1-6]
\arrow["\sigma"', from=3-5, to=3-7]
\end{tikzcd}
\end{equation}

\vspace{0.2cm}

\textbf{(III)} Let $X$ be a zero-dimensional (that is, the topology of $X$ contain a base of clopen sets) Hausdorff space and let $K$ be a zero-dimensional compact Hausdorff space. Let us now consider the associated transition maps $Z_{1}: C_{c}(K) \rightarrow \powerset(K)$ and $Z_{2}: C_{c}(X) \rightarrow \powerset(X)$ defined by:
\begin{align*}
Z_{1}(\alpha) = \{x \in K : \alpha(x) = 0\} \text{ and } Z_{2}(\beta) = \{x \in X : \beta(x) = 0\}.
\end{align*}
If $U \subseteq K$ is clopen, then its characteristic function $\chi_{U}$ defined by:
$$\chi_{U}(x) = 
\begin{cases} 0 & \text{if } x \in U \\ 
1 & \text{if } x \notin U 
\end{cases}$$
is an element of $C_{c}(K)$. Hence, the topology generated by the closed base $\{Z_{1}(\alpha) \mid \alpha \in C_{c}(K)\}$ is equal to the original topology on $K$. Thus, $K$ is compact with respect to the topology which is generated by the closed base $\{Z_{1}(\alpha) \mid \alpha \in C_{c}(K)\}$. Let $\sigma: X \rightarrow K$ be any continuous function. Then there is a ring homomorphism $\varphi: C_{c}(K) \rightarrow C_{c}(X)$, defined by $\varphi(\alpha) = \alpha \circ \sigma$, for all $\alpha \in C_{c}(K)$. Now, observe that we have $\sigma^{-1}(Z_{1}(\alpha)) = Z_{2}(\alpha \circ \sigma) = Z_{2}(\varphi(\alpha))$. This shows that the first diagram in \eqref{MAHARAJJIJOY} commutes. Hence, applying Theorem \eqref{ramramram}, we conclude that $\sigma$ has a continuous extension $\overline{\sigma}: \Omega_{f}(X) \rightarrow K$ such that the second diagram in \eqref{MAHARAJJIJOY} also commutes. Now, we can conclude from \cite{p21} that $\Omega_{f}(X) \cong \beta_{0}X$, the Banaschewski compactification of $X$.
\begin{equation}\label{MAHARAJJIJOY}
\begin{tikzcd}
{C_{c}(K)} && {Z_{1}(C_{c}(K))} &&& {\Omega_{f}(X)} \\
\\
{C_{c}(X)} && {Z_{2}(C_{c}(X))} && X && K
\arrow["Z_{1}", from=1-1, to=1-3]
\arrow["\varphi"', from=1-1, to=3-1]
\arrow["{\sigma^{\ast}}", from=1-3, to=3-3]
\arrow["{\overline{\sigma}}", dashed, from=1-6, to=3-7]
\arrow["Z_{2}"', from=3-1, to=3-3]
\arrow["i", from=3-5, to=1-6]
\arrow["\sigma"', from=3-5, to=3-7]
\end{tikzcd}
\end{equation}

\bibliographystyle{amsplain}

\end{document}